\documentclass[runningheads]{llncs}

\title{Automatic Verification of \\ Floating-Point Accumulation Networks}

\author{
    David K.~Zhang \orcidID{0000-0003-3379-4890} \and
    Alex Aiken \orcidID{0000-0002-3723-9555}
}
\authorrunning{D. K. Zhang \and A. Aiken}

\institute{Stanford University}

\usepackage[boxruled,linesnumbered]{algorithm2e}
\usepackage{amssymb}
\usepackage{bm}
\usepackage{float}
\usepackage[hidelinks]{hyperref}
\usepackage{mathtools}
\usepackage{tikz}
\usepackage{xcolor}

\usetikzlibrary{decorations.pathreplacing}

\newcommand{\RNE}{\operatorname{\mathsf{RNE}}}
\newcommand{\ulp}{\operatorname{\mathsf{ulp}}}
\newcommand{\nlz}{\mathsf{nlz}}
\newcommand{\nlo}{\mathsf{nlo}}
\newcommand{\ntz}{\mathsf{ntz}}
\newcommand{\nto}{\mathsf{nto}}
\renewcommand{\u}{\bm{\mathsf{u}}}
\newcommand{\TwoSum}{\operatorname{\mathsf{TwoSum}}}
\newcommand{\emin}{e_{\mathrm{min}}}
\newcommand{\emax}{e_{\mathrm{max}}}

\begin{document}
\maketitle

\begin{abstract}
    Floating-point accumulation networks (FPANs) are key building blocks used in many floating-point algorithms, including compensated summation and double-double arithmetic.
    FPANs are notoriously difficult to analyze, and algorithms using FPANs are often published without rigorous correctness proofs.
    In fact, on at least one occasion, a published error bound for a widely used FPAN was later found to be incorrect.
    In this paper, we present an automatic procedure that produces computer-verified proofs of several FPAN correctness properties, including error bounds that are tight to the nearest bit.
    Our approach is underpinned by a novel floating-point abstraction that models the sign, exponent, and number of leading and trailing zeros and ones in the mantissa of each number flowing through an FPAN.
    We also present a new FPAN for double-double addition that is faster and more accurate than the previous best known algorithm.
\end{abstract}

\section{Introduction} \label{sec:Introduction}

Many scientific and mathematical problems demand calculations that exceed the native precision limits of floating-point hardware (typically IEEE 64-bit or Intel x87 80-bit). To address this need, numerical programmers turn to \textit{compensated algorithms}, such as Kahan--Babu\v{s}ka--Neumaier summation~\cite{Kahan1965,Babuska1972,Neumaier1974}, and \textit{floating-point expansions}~\cite{Priest1991}, such as double-double and quad-double arithmetic~\cite{Dekker1971,QD}. These techniques enable a processor to execute floating-point operations with double, quadruple, or even higher multiples of its native precision. They are widely adopted in dense~\cite{XBLAS} and sparse~\cite{Evstigneev2022} numerical linear algebra, high-precision quadrature~\cite{Bailey2012}, robust computational geometry~\cite{Shewchuk1997}, fluid dynamics~\cite{Bailey1993,He2001}, quantum chemistry~\cite{Frolov2000,Frolov2003}, correctly rounded transcendental functions~\cite{CRLibM,COREMATH}, and the discovery of new mathematical identities~\cite{Bailey2005}.

All of these techniques are based on common building blocks known as \textit{error-free transformations}~\cite{Rump2005}, which are floating-point algorithms that exactly compute their own rounding errors. For example, the M{\o}ller--Knuth $\TwoSum$ algorithm~\cite{Moller1965,Knuth1969} takes a pair of floating-point numbers $(x, y)$ and computes both their rounded sum $s \coloneqq x \oplus y$ and the exact rounding error $e \coloneqq (x + y) - (x \oplus y)$ incurred in that sum. Here, $\oplus$ denotes rounded floating-point addition, while $+$ denotes true mathematical addition.

Although the $\TwoSum$ algorithm has been extensively studied, proven correct by both pen-and-paper~\cite{Knuth1969,Boldo2017Robustness} and formal methods~\cite{Boldo2017Formal,Muller2022}, issues arise when multiple $\TwoSum$ operations are combined to accumulate three or more floating-point values. Multiple rounding error terms can interact in subtle and unexpected ways, significantly affecting the precision of the overall result when terms that are incorrectly assumed to be negligible are discarded.

\begin{figure}[t]
    \begin{minipage}{0.52\textwidth}
        \centering
        \begin{tikzpicture}[scale=0.5]
            \draw (0,4) node [anchor=east] {$x_0$} -- (10,4) node [anchor=west] {$z_0$};
            \draw (0,3) node [anchor=east] {$y_0$} -- (10,3) node [anchor=west] {$z_1$};
            \draw (0,2) node [anchor=east] {$x_1$} -- (3.75,2);
            \draw (0,1) node [anchor=east] {$y_1$} -- (7.75,1);
    
            \draw (1,4) circle (0.25);
            \draw (1,3) circle (0.25);
            \draw (1,4.25) -- (1,2.75);
    
            \draw (1,2) circle (0.25);
            \draw (1,1) circle (0.25);
            \draw (1,2.25) -- (1,0.75);
    
            \draw (3,3) circle (0.25);
            \draw (3,2) circle (0.25);
            \draw (3,3.25) -- (3,1.75);
    
            \draw (3.75,1.75) rectangle (4.25,2.25);
            \draw (3.75,1.75) -- (4.25,2.25);
            \draw (3.75,2.25) -- (4.25,1.75);
    
            \draw (5,4) circle (0.25);
            \draw (5,3) circle (0.25);
            \draw (5,4.25) -- (5,2.75);
    
            \draw (7,3) circle (0.25);
            \draw (7,1) circle (0.25);
            \draw (7,3.25) -- (7,0.75);
    
            \draw (7.75,0.75) rectangle (8.25,1.25);
            \draw (7.75,0.75) -- (8.25,1.25);
            \draw (7.75,1.25) -- (8.25,0.75);
    
            \draw (9,4) circle (0.25);
            \draw (9,3) circle (0.25);
            \draw (9,4.25) -- (9,2.75);
        \end{tikzpicture}
    \end{minipage}\hfill\vline\hfill\begin{minipage}{0.47\textwidth}
        \centering
        \begin{tikzpicture}[scale=0.5]
            \draw (0,4) node [anchor=east] {$x_0$} -- (9,4) node [anchor=west] {$z_0$};
            \draw (0,3) node [anchor=east] {$y_0$} -- (9,3) node [anchor=west] {$z_1$};
            \draw (0,2) node [anchor=east] {$x_1$} -- (6.75,2);
            \draw (0,1) node [anchor=east] {$y_1$} -- (4.75,1);
    
            \draw (1,4) circle (0.25);
            \draw (1,3) circle (0.25);
            \draw (1,4.25) -- (1,2.75);
    
            \draw (1,2) circle (0.25);
            \draw (1,1) circle (0.25);
            \draw (1,2.25) -- (1,0.75);
    
            \draw (3,4) circle (0.25);
            \draw (3,2) circle (0.25);
            \draw (3,4.25) -- (3,1.75);
    
            \draw (4,3) circle (0.25);
            \draw (4,1) circle (0.25);
            \draw (4,3.25) -- (4,0.75);
    
            \draw (4.75,0.75) rectangle (5.25,1.25);
            \draw (4.75,0.75) -- (5.25,1.25);
            \draw (4.75,1.25) -- (5.25,0.75);
    
            \draw (6,3) circle (0.25);
            \draw (6,2) circle (0.25);
            \draw (6,3.25) -- (6,1.75);
    
            \draw (6.75,1.75) rectangle (7.25,2.25);
            \draw (6.75,1.75) -- (7.25,2.25);
            \draw (6.75,2.25) -- (7.25,1.75);
    
            \draw (8,4) circle (0.25);
            \draw (8,3) circle (0.25);
            \draw (8,4.25) -- (8,2.75);
        \end{tikzpicture}
    \end{minipage}
    \caption{Network diagram of the \textsf{ddadd} algorithm due to Li et al.\ (left) and \textsf{madd}, our new double-double addition algorithm (right). This graphical FPAN notation will be explained in detail in Section~\ref{sec:FPANs}.}
    \label{fig:networks}
\end{figure}
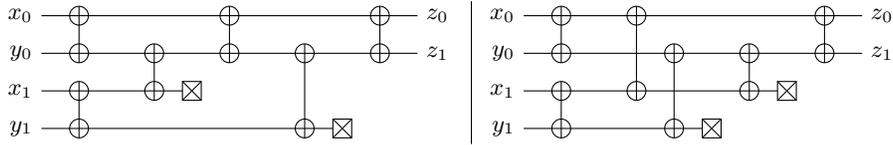

In 2017, Joldes, Muller, and Popescu~\cite{Joldes2017} discovered an error of this type in the \textsf{ddadd} algorithm designed by Li et al.\ for the XBLAS extended-precision linear algebra library~\cite{XBLAS}. Li et al.\ originally claimed that \textsf{ddadd} computes sums with relative error bounded by $2 \cdot 2^{-106} = 2\u^2$ when executed in IEEE \textsf{binary64} (double precision) arithmetic, where $\u = 2^{-53}$ denotes the unit roundoff. Joldes, Muller, and Popescu refuted this claim by identifying a class of inputs in which two discarded error terms interfere constructively, producing a relative error of $2.25 \cdot 2^{-106} = 2.25\u^2$ in the final sum. While this weakened bound does not invalidate the usefulness of XBLAS, it demonstrates that even expert numerical analysts can make subtle mistakes in the analysis of rounding errors.

The \textsf{ddadd} algorithm is a prototypical example of a \textit{floating-point accumulation network (FPAN)} --- a term we introduce to describe a branch-free linear sequence of floating-point sum and $\TwoSum$ operations. Although this class of algorithms, to our knowledge, has never been explicitly named in the floating-point literature, FPANs pervade floating-point algorithms and occur in every single paper referenced in this introduction, along with dozens of software packages~\cite{PythonDocs,JuliaDocs,CRLibM,COREMATH,QD,GQD,AccurateArithmetic,CAMPARY,DoubleFloats,MultiFloats,BaileySoftware}. In addition to their obvious uses for high-precision addition and subtraction, FPANs also occur as subroutines in multiplication, divison, and square root algorithms~\cite{KarpMarkstein1997,QD}, which in turn are used to implement transcendental functions, including the exponential, logarithm, and trigonometric functions, in many standard libraries~\cite{CRLibM,COREMATH,JuliaDocs}.

In this paper, we present an automatic procedure for proving several FPAN correctness properties, including error bounds and nonoverlapping invariants. Our key technical insight is a one-sided reduction from a correctness property $P$ of an FPAN $F$ to a satisfiability problem $S_{P,F}$ in quantifier-free Presburger arithmetic (\texttt{QF{\textunderscore}LIA}). Here, one-sidedness means that if $S_{P,F}$ is unsatisfiable, then $P(F)$ provably holds, but if $S_{P,F}$ is satisfiable, then $P(F)$ may or may not hold. In all cases we have tested, we find that the satisfiability of $S_{P,F}$ is far easier for SMT solvers to determine than direct verification of $P(F)$ in a floating-point theory (\texttt{QF{\textunderscore}FP}). Our reduced problems $S_{P,F}$ can be solved in less than one second, while direct $P(F)$ verifiers fail to terminate after multiple days of runtime.

Using our procedure, we discover and prove the correctness of a new algorithm for double-double addition --- the same problem that \textsf{ddadd} solves --- with strictly smaller relative error and lower circuit depth. Thus, our algorithm, named \textsf{madd} (for ``More Accurate Double-Double addition,'' shown in Figure~\ref{fig:networks}) is both faster and more accurate than the previous best known algorithm for this task. We expect our automatic decision procedure to enable additional novel algorithms to be found by computer search in future work.

Our reduction strategy uses a novel abstraction of floating-point arithmetic that does not consider the full bitwise representation of a floating-point number $x$, but tracks only its sign $s_x$, exponent $e_x$, and the number of leading and trailing zeros and ones in its mantissa $(\nlz_x, \nlo_x, \ntz_x, \nto_x)$. Using this reduced representation, we construct abstract overapproximations of the floating-point sum and $\TwoSum$ operations, expressed as linear inequalities in the variables $(s_x, e_x, \nlz_x, \nlo_x, \ntz_x, \nto_x)$. Although we lose completeness by passing to an overapproximation, we show that our abstraction is nonetheless precise enough to prove best-possible error bounds, tight to the nearest bit, for several FPANs of practical interest, including \textsf{ddadd} and \textsf{madd}. Moreover, we use floating-point SMT solvers to directly verify the soundness of our abstraction.

In summary, this paper makes the following contributions:
\begin{enumerate}
    \item We define \textit{floating-point accumulation networks} (\textit{FPANs}) and show that FPANs occur in many widely used floating-point algorithms.
    \item We give a procedure for reducing correctness properties $P$ of an FPAN $F$ to satisfiability problems $S_{P,F}$ in quantifier-free Presburger arithmetic. Our strategy uses a novel formally verified abstraction of floating-point arithmetic that may be of independent interest.
    \item We empirically demonstrate that $S_{P,F}$ is far easier for modern SMT solvers to reason about than $P(F)$, solving otherwise intractable verification problems in under one second.
    \item We state a new double-double addition algorithm, \textsf{madd}, with lower circuit depth and strictly smaller relative error than the previous best known algorithm, proven by our reduction procedure.
\end{enumerate}

\section{Background} \label{sec:Background}

\subsection{Floating-Point Numbers}

A \textit{floating-point number} in base $b \in \mathbb{N}$ with precision $p \in \mathbb{N}$ is an ordered triple $(s, e, m)$ consisting of a \textit{sign bit} $s \in \{0, 1\}$, an \textit{exponent} $e \in \mathbb{Z}$, and a \textit{mantissa} $m = (m_0, m_1, \dots, m_{p-1})$, which is an sequence of $p$ \textit{digits} $m_k \in \{0, 1, \dots, b - 1\}$. The value represented by $(s, e, m)$ is defined as the following real number:
\begin{equation}
(-1)^s \times (m_0 . m_1 m_2 \dots m_{p-1}) \times b^e
= (-1)^s \sum_{k=0}^{p-1} m_k b^{e-k}
\end{equation}
Floating-point hardware is almost always designed for base $b = 2$, and we assume $b = 2$ throughout this paper.

The IEEE 754 standard~\cite{IEEE7541985,IEEE7542008,IEEE7542019} defines a family of encodings of floating-point numbers $(s, e, m)$ into fixed-width bit-vectors with a specified base $b$, precision $p$, and exponent range $\emin \le e \le \emax$. These standard encodings, listed in Table~\ref{tab:ieeeformats}, have several additional features that are relevant to our analysis.

\begin{itemize}
    \item A floating-point number with $m_0 = 1$ is said to be \textit{normalized}. A nonzero floating-point number can be normalized without changing its represented real value by shifting its first nonzero bit $m_k$ to position $m_0$ and adjusting its exponent $e$ to $e - k$. IEEE 754 requires floating-point numbers to be normalized whenever possible so that the \textit{implicit leading bit} $m_0 = 1$ does not need to be explicitly stored.
    \item The only floating-point numbers that cannot be normalized are zero and very small numbers whose adjusted exponent would fall below the minimum threshold $\emin$. These small nonzero numbers are called \textit{subnormal} and use a special alternative representation with no implicit leading bit.
    \item Zero has two distinct IEEE 754 encodings with opposite sign bits, denoted by \texttt{+0.0} (\textit{positive zero}) and \texttt{-0.0} (\textit{negative zero}). These two values are considered to be equal, i.e., \texttt{+0.0 == -0.0} evaluates to \texttt{true} in any IEEE 754-compliant programming environment.
    \item IEEE 754 defines three non-numeric floating-point values, called \textit{positive infinity} ($+\infty$), \textit{negative infinity} ($-\infty$), and \textit{not-a-number} (\texttt{NaN}). These special values are returned from operations that would otherwise yield an unrepresentably large or indeterminate result. All floating-point values other than $\pm\infty$ and \texttt{NaN} are called \textit{finite}.
\end{itemize}

We assume throughout this paper that all floating-point numbers are normalized or zero, ignoring subnormal numbers, $\pm\infty$, and \texttt{NaN}. We also identify \texttt{+0.0} with \texttt{-0.0}, writing $\pm\mathtt{0.0}$ to denote zero with either sign. These assumptions serve only to simplify exposition and do not affect the generality of our results. Indeed, we prove in Section~\ref{sec:Verification} that our automated proof technique handles all finite floating-point values, including subnormal numbers.

\begin{table}[t]
    \centering
    \setlength\tabcolsep{6pt}
    \begin{tabular}{|c|c|c|c|} \hline
        Format & Base $b$ & Precision $p$ & Exponent range $\{\emin, \dots, \emax\}$ \\\hline\hline
        \textsf{binary16}  & $b = 2$ & $p =  11$ & $e \in \{    -14, \dots,     +15\}$ \\\hline
        \textsf{bfloat16}  & $b = 2$ & $p =   8$ & $e \in \{   -126, \dots,    +127\}$ \\\hline
        \textsf{binary32}  & $b = 2$ & $p =  24$ & $e \in \{   -126, \dots,    +127\}$ \\\hline
        \textsf{binary64}  & $b = 2$ & $p =  53$ & $e \in \{  -1022, \dots,   +1023\}$ \\\hline
        \textsf{binary128} & $b = 2$ & $p = 113$ & $e \in \{ -16382, \dots,  +16383\}$ \\\hline
    \end{tabular}
    \vspace{\baselineskip}
    \caption{Parameters of the floating-point formats defined by the IEEE 754 standard and the nonstandard \textsf{bfloat16} format commonly used in deep learning accelerators.}
    \label{tab:ieeeformats}
\end{table}

\subsection{Rounding and Error-Free Transformations}

In general, the sum or difference of two precision-$p$ floating-point numbers may not be exactly representable as another precision-$p$ floating-point number. To perform calculations at a fixed precision $p$, the result of every operation must be \textit{rounded} by the following procedure:
\begin{enumerate}
    \item Compute the exact sum or difference of the real values represented by the floating-point inputs, as if to infinite precision.
    \item Find and return the closest precision-$p$ floating-point number to the exact result. A \textit{tie-breaking} rule must be used whenever the exact result is equidistant to two neighboring floating-point values.
\end{enumerate}
As is standard in studies of extended-precision floating-point arithmetic \cite{Boldo2017Robustness,Boldo2017Formal}, we assume throughout this paper that all floating-point operations are rounded to nearest with ties broken to even, denoted by $\RNE$. This is the default rounding mode defined by IEEE 754 and is supported on virtually all general-purpose computing hardware. In fact, $\RNE$ is the only rounding mode available in many programming environments, including Python, JavaScript, and WebAssembly~\cite{ECMAScript,WebAssembly}. Following Knuth's convention \cite{Knuth1969}, we distinguish between rounded and exact arithmetic operations using circled operators.
\begin{align}
    x \oplus y &\coloneqq \RNE(x + y) \\
    x \ominus y &\coloneqq \RNE(x - y)
\end{align}

Rounding errors in precision-$p$ floating-point arithmetic are characterized by the \textit{unit roundoff}\footnote{The definition $\u \coloneqq 2^{-p}$ is appropriate when floating-point operations are rounded to nearest. For other rounding modes, the larger value $\u \coloneqq 2^{-(p-1)}$ is used instead.} constant $\u \coloneqq 2^{-p}$, which bounds the relative error of any individual rounded operation \cite{Rump2016}. For example, for all floating-point numbers $a$ and $b$ such that $|a + b| < (2 - \u) 2^{\emax}$, there exists $\delta \in \mathbb{R}$ satisfying
\begin{equation}
a \oplus b = (a + b)(1 + \delta) \quad \text{where} \quad |\delta| \le \u.
\end{equation}
Analogous results also hold for subtraction and other floating-point operations.

An \textit{error-free transformation} is a floating-point algorithm that computes a rounded arithmetic operation together with the exact rounding error incurred by that operation. The oldest and most important example of an error-free transformation is the M{\o}ller--Knuth $\TwoSum$ algorithm, first discovered by M{\o}ller \cite{Moller1965} and proven correct by Knuth \cite{Knuth1969}, which computes the rounding error of a floating-point sum $s \coloneqq x \oplus y$ or difference $d \coloneqq x \ominus y$.

Although the algorithm itself is straightforward, the existence of $\TwoSum$ is a highly nontrivial result. It is not obvious \textit{a priori} that the rounding error $e \coloneqq (x + y) - (x \oplus y)$ is always exactly representable as a floating-point number, let alone that it can be recovered using rounded floating-point operations. Despite the apparent simplicity of the pseudocode in Figure~\ref{fig:TwoSum}, proving the correctness of $\TwoSum$ requires lengthy case analysis \cite[Sec. 4.2.2]{Knuth1969}. Analogous error-free transformations also exist for floating-point multiplication~\cite{Veltkamp1968,Veltkamp1969,Dekker1971} and the fused multiply-add operation~\cite{Boldo2011}.

\begin{figure}[t]
    \centering
    \begin{minipage}{8.6cm}
        \begin{algorithm}[H]
            \caption{$\TwoSum(x, y)$}
            \KwIn{Two floating-point numbers $(x, y)$.}
            \KwOut{Two floating-point numbers $(s, e)$
                such that $s = x \oplus y$ and
                $e = (x + y) - (x \oplus y)$ exactly.}
            $s \coloneqq x \oplus y$ \\
            $x_{\text{eff}} \coloneqq s \ominus y$ \\
            $y_{\text{eff}} \coloneqq s \ominus x_{\text{eff}}$ \\
            $\delta_x \coloneqq x \ominus x_{\text{eff}}$ \\
            $\delta_y \coloneqq y \ominus y_{\text{eff}}$ \\
            $e \coloneqq \delta_x \oplus \delta_y$ \\
            \KwRet{$(s, e)$}
        \end{algorithm}
    \end{minipage}
    \caption{Pseudocode for the M{\o}ller--Knuth $\TwoSum$ algorithm. This algorithm can also be applied to subtraction by flipping the sign of $y$ throughout.}
    \label{fig:TwoSum}
\end{figure}

\subsection{Floating-Point Expansions}

\textit{Floating-point expansion} is a technique for representing high-precision numbers as sequences of machine-precision numbers. The basic idea is to represent a high-precision constant $C \in \mathbb{R}$ as a sequence of successive approximations:
\begin{equation}
    \begin{aligned}
        x_0 &\coloneqq \RNE(C) \\
        x_1 &\coloneqq \RNE(C - x_0) \\
        x_2 &\coloneqq \RNE(C - x_0 - x_1) \\
        &\hspace{0.7em} \vdots \\
        x_{n-1} &\coloneqq \RNE(C - x_0 - x_1 - \cdots - x_{n-2})
    \end{aligned}    
\end{equation}
Provided that no overflow or underflow occurs in this process, the final $n$-term expansion $(x_0, x_1, \dots, x_{n-1})$ approximates $C$ with precision $np + n - 1$, i.e.,
\begin{equation}
    |C - (x_0 + x_1 + \cdots + x_{n-1})| \le 2^{-(np+n-1)} |C|
\end{equation}
where $p$ denotes the underlying machine precision. To achieve the full precision of $np + n - 1$ bits, a floating-point expansion must be \textit{nonoverlapping}, i.e.,
\begin{equation} \label{eq:QDNO}
    |x_i| \le \frac{1}{2} \ulp(x_{i-1})
\end{equation}
for each $i = 1, \dots, n-1$. Here, $\ulp(x)$ denotes a \textit{unit in the last place} of $x$, i.e., the place value of its least significant mantissa bit, in accordance with Goldberg's convention~\cite{MullerUlp}. As illustrated in Figure~\ref{fig:overlap}, this property ensures that no bits of $C$ are redundantly covered by more than one component of the expansion.

\begin{figure}[t]
    \centering
    \begin{tikzpicture}[scale=0.25,decoration=brace]

        \draw (-5,10.85) node [anchor=east] {high-precision constant};
        \draw (-3,10) node [anchor=south] {$C$};
        \draw (-1.5,10) node [anchor=south] {$=$};
        \draw ( 0,10) node [anchor=south] {1};
        \draw ( 1,10) node [anchor=south] {0};
        \draw ( 2,10) node [anchor=south] {1};
        \draw ( 3,10) node [anchor=south] {.};
        \draw ( 4,10) node [anchor=south] {0};
        \draw ( 5,10) node [anchor=south] {1};
        \draw ( 6,10) node [anchor=south] {1};
        \draw ( 7,10) node [anchor=south] {1};
        \draw ( 8,10) node [anchor=south] {0};
        \draw ( 9,10) node [anchor=south] {1};
        \draw (10,10) node [anchor=south] {1};
        \draw (11,10) node [anchor=south] {0};
        \draw (12,10) node [anchor=south] {1};
        \draw (13,10) node [anchor=south] {0};
        \draw (14,10) node [anchor=south] {1};
        \draw (15,10) node [anchor=south] {1};
        \draw (16,10) node [anchor=south] {.};
        \draw (17,10) node [anchor=south] {.};
        \draw (18,10) node [anchor=south] {.};

        \draw (-5,7.75) node [anchor=east] {expansion with $|x_1| > \ulp(x_0)$};
        \draw (12,7.75) node [anchor=west] {10-bit precision};
        \draw[decorate] (-4.5,6.25) -- (-4.5,9.25);
        \draw[decorate] (11,9.25) -- (11,6.25);

        \draw (-3,7.5) node [anchor=south] {$x_0$};
        \draw (-1.5,7.5) node [anchor=south] {$=$};
        \draw ( 0,7.5) node [anchor=south] {1};
        \draw ( 1,7.5) node [anchor=south] {0};
        \draw ( 2,7.5) node [anchor=south] {1};
        \draw ( 3,7.5) node [anchor=south] {.};
        \draw ( 4,7.5) node [anchor=south] {0};
        \draw ( 5,7.5) node [anchor=south] {0};
        \draw ( 6,7.5) node [anchor=south] {0};

        \draw (-3,6) node [anchor=south] {$x_1$};
        \draw (-1.5,6) node [anchor=south] {$=$};
        \draw[blue!15] ( 2,6) node [anchor=south] {0};
        \draw[blue!15] ( 3,6) node [anchor=south] {.};
        \draw[blue!15] ( 4,6) node [anchor=south] {0};
        \draw ( 5,6) node [anchor=south] {1};
        \draw ( 6,6) node [anchor=south] {1};
        \draw ( 7,6) node [anchor=south] {1};
        \draw ( 8,6) node [anchor=south] {0};
        \draw ( 9,6) node [anchor=south] {1};
        \draw (10,6) node [anchor=south] {1};

        \draw (-5,3.75) node [anchor=east] {expansion with $|x_1| \le \ulp(x_0)$};
        \draw (14,3.75) node [anchor=west] {12-bit precision};
        \draw[decorate] (-4.5,2.25) -- (-4.5,5.25);
        \draw[decorate] (13,5.25) -- (13,2.25);

        \draw (-3,3.5) node [anchor=south] {$x_0$};
        \draw (-1.5,3.5) node [anchor=south] {$=$};
        \draw ( 0,3.5) node [anchor=south] {1};
        \draw ( 1,3.5) node [anchor=south] {0};
        \draw ( 2,3.5) node [anchor=south] {1};
        \draw ( 3,3.5) node [anchor=south] {.};
        \draw ( 4,3.5) node [anchor=south] {0};
        \draw ( 5,3.5) node [anchor=south] {1};
        \draw ( 6,3.5) node [anchor=south] {1};

        \draw (-3,2) node [anchor=south] {$x_1$};
        \draw (-1.5,2) node [anchor=south] {$=$};
        \draw[blue!15] ( 2,2) node [anchor=south] {0};
        \draw[blue!15] ( 3,2) node [anchor=south] {.};
        \draw[blue!15] ( 4,2) node [anchor=south] {0};
        \draw[blue!15] ( 5,2) node [anchor=south] {0};
        \draw[blue!15] ( 6,2) node [anchor=south] {0};
        \draw ( 7,2) node [anchor=south] {1};
        \draw ( 8,2) node [anchor=south] {0};
        \draw ( 9,2) node [anchor=south] {1};
        \draw (10,2) node [anchor=south] {1};
        \draw (11,2) node [anchor=south] {0};
        \draw (12,2) node [anchor=south] {1};

        \draw (-5,-0.25) node [anchor=east] {expansion with $|x_1| \le \frac{1}{2} \ulp(x_0)$};
        \draw (15,-0.25) node [anchor=west] {13-bit precision};
        \draw[decorate] (-4.5,-1.75) -- (-4.5,1.25);
        \draw[decorate] (14,1.25) -- (14,-1.75);

        \draw (-3,-0.5) node [anchor=south] {$x_0$};
        \draw (-1.5,-0.5) node [anchor=south] {$=$};
        \draw ( 0,-0.5) node [anchor=south] {1};
        \draw ( 1,-0.5) node [anchor=south] {0};
        \draw ( 2,-0.5) node [anchor=south] {1};
        \draw ( 3,-0.5) node [anchor=south] {.};
        \draw ( 4,-0.5) node [anchor=south] {1};
        \draw ( 5,-0.5) node [anchor=south] {0};
        \draw ( 6,-0.5) node [anchor=south] {0};

        \draw (-3,-2) node [anchor=south] {$x_1$};
        \draw (-1.5,-2) node [anchor=south] {$=$};
        \draw ( 1,-2) node [anchor=south] {$-$};
        \draw[blue!15] ( 2,-2) node [anchor=south] {0};
        \draw[blue!15] ( 3,-2) node [anchor=south] {.};
        \draw[blue!15] ( 4,-2) node [anchor=south] {0};
        \draw[blue!15] ( 5,-2) node [anchor=south] {0};
        \draw[blue!15] ( 6,-2) node [anchor=south] {0};
        \draw[blue!15] ( 7,-2) node [anchor=south] {0};
        \draw ( 8,-2) node [anchor=south] {1};
        \draw ( 9,-2) node [anchor=south] {0};
        \draw (10,-2) node [anchor=south] {0};
        \draw (11,-2) node [anchor=south] {1};
        \draw (12,-2) node [anchor=south] {0};
        \draw (13,-2) node [anchor=south] {1};
    \end{tikzpicture}
    \caption{Decomposition of a high-precision constant $C$ into overlapping and nonoverlapping floating-point expansions with $p = 6$ mantissa bits per term. Light blue digits represent a shift stored in the exponent and are not part of the mantissa. Note that the final expansion rounds $x_0$ up instead of down. In this case, $x_1$ is negative, and the mantissa of $x_1$ contains the one's complement of the corresponding digits in $C$.}
    \label{fig:overlap}
\end{figure}

Arithmetic with floating-point expansions is a delicate procedure that requires skillful use of error-free transformations to propagate rounding errors between terms. Moreover, the nonoverlapping property is easily broken and must be restored after every few operations to avoid loss of precision. Designing sequences of error-free transformations with correct error propagation and nonoverlapping semantics is a remarkably difficult problem; the literature on this subject is punctuated by refutations, corrections, and corrections to those corrections~\cite{Joldes2017,Muller2022}. Some general constructions are known, but these algorithms are far from optimal, particularly when the number of inputs is small~\cite{Collange2016,FPHandbook}.

In principle, a nonoverlapping floating-point expansion can contain up to $\lceil (\emax - \emin + p) / (p + 1) \rceil$ terms before underflow occurs, causing all subsequent terms to round down to zero. However, in practice, it is more common to use fixed-length floating-point expansions consisting of two, three, or four terms~\cite{XBLAS,Lauter2005,CRLibM,COREMATH,QD}. These fixed-length expansions are called \textit{double-word}, \textit{triple-word}, \textit{quadruple-word} or \textit{double-double}, \textit{triple-double}, \textit{quad-double} representations. The latter names are used when the underlying machine format is IEEE \textsf{binary64} (double precision), but most algorithms for double-double, triple-double, and quad-double arithmetic also work for other underlying formats.

\section{Floating-Point Accumulation Networks} \label{sec:FPANs}

In this section, we formally define floating-point accumulation networks as a class of branch-free floating-point algorithms using a graphical notation inspired by sorting networks~\cite{Knuth1973}. \\[-\baselineskip]

\begin{definition}
    A \textbf{floating-point accumulation network} (\textbf{FPAN}) is a diagram consisting of horizontal \textbf{wires} and vertical \textbf{$\TwoSum$ gates} that connect exactly two wires. Each wire may optionally be terminated by the symbol $\boxtimes$, indicating that it is \textbf{discarded}.
\end{definition}
\begin{figure}[H]
    \centering
    \begin{tikzpicture}[scale=0.75]
        \draw (0.25,2) node [anchor=east] {$x$} -- (1.75,2) node [anchor=west] {$s$};
        \draw (0.25,1) node [anchor=east] {$y$} -- (1.75,1) node [anchor=west] {$e$};
        \draw (1,1) circle (0.25);
        \draw (1,2) circle (0.25);
        \draw (1,0.75) -- (1,2.25);
        \draw (2.5,1.5) node [anchor=west] {$(s, e) \coloneqq \TwoSum(x, y)$};
        \draw (8.25,2) node [anchor=east] {$x$} -- (9.75,2) node [anchor=west] {$s$};
        \draw (8.25,1) node [anchor=east] {$y$} -- (9.75,1);
        \draw (9,1) circle (0.25);
        \draw (9,2) circle (0.25);
        \draw (9,0.75) -- (9,2.25);
        \draw (9.75,0.75) rectangle (10.25,1.25);
        \draw (9.75,0.75) -- (10.25,1.25);
        \draw (9.75,1.25) -- (10.25,0.75);
        \draw (10.5,1.5) node [anchor=west] {$s \coloneqq x \oplus y$};
    \end{tikzpicture}
    \caption{Gate representations of the floating-point sum and $\TwoSum$ operations. In our notation, $x \oplus y$ is treated as a special case of $\TwoSum(x, y)$ with the error term discarded. Note that the inputs are unordered (i.e., $\TwoSum(x, y) = \TwoSum(y, x)$) but the order of the outputs is significant, with larger-magnitude outputs on top.}
    \label{fig:TwoSumGate}
\end{figure}
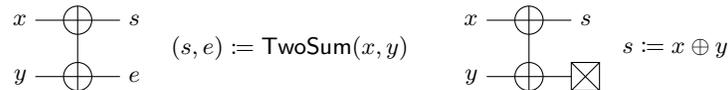

An FPAN with $n$ wires, of which $k$ are discarded, represents the following algorithm with $n$ floating-point inputs and $n-k$ floating-point outputs. Each input value $(x_1, x_2, \dots, x_n)$ enters on the left-hand side of each wire, ordered top-to-bottom unless otherwise specified by explicit labels. The values flow left-to-right along the wires, and whenever two values $(x_i, x_j)$ encounter a $\TwoSum$ gate, they are updated by $(x_i, x_j) \gets \TwoSum(x_i, x_j)$ as specified in Figure~\ref{fig:TwoSumGate}. After all $\TwoSum$ gates have been executed, all values on wires that are not discarded are returned in top-to-bottom order. To illustrate this definition, Figure~\ref{fig:add2} presents equivalent pseudocode and network diagram representations of Dekker's \textsf{add2} algorithm, the first algorithm ever proposed for double-double addition~\cite{Dekker1971}.

\begin{figure}[t]
    \begin{minipage}{0.5\textwidth}
        \centering
        \begin{algorithm}[H]
            \caption{$\mathrlap{\mathsf{add2}((x_0, x_1), (y_0, y_1))}$}
            \KwIn{nonoverlapping expansions $(x_0, x_1)$ and $(y_0, y_1)$.}
            \KwOut{nonoverlapping expansion $(z_0, z_1)$ for $x + y$.}
            $(s_0, s_1) \coloneqq \TwoSum(x_0, y_0)$ \\
            $t \coloneqq x_1 \oplus y_1$ \\
            $u \coloneqq s_1 \oplus t$ \\
            $(z_0, z_1) \coloneqq \TwoSum(s_0, u)$ \\
            \KwRet{$(z_0, z_1)$}
        \end{algorithm}
    \end{minipage}\begin{minipage}{0.5\textwidth}
        \centering
        \begin{tikzpicture}[scale=0.7]
            \draw (0,4) node [anchor=east] {$x_0$} -- (6,4) node [anchor=west] {$z_0$};
            \draw (0,3) node [anchor=east] {$y_0$} -- (6,3) node [anchor=west] {$z_1$};
            \draw (0,2) node [anchor=east] {$x_1$} -- (3.75,2);
            \draw (0,1) node [anchor=east] {$y_1$} -- (1.75,1);
    
            \draw (1,4) circle (0.25);
            \draw (1,3) circle (0.25);
            \draw (1,4.25) -- (1,2.75);
    
            \draw (1,2) circle (0.25);
            \draw (1,1) circle (0.25);
            \draw (1,2.25) -- (1,0.75);
    
            \draw (1.75,0.75) rectangle (2.25,1.25);
            \draw (1.75,0.75) -- (2.25,1.25);
            \draw (1.75,1.25) -- (2.25,0.75);
    
            \draw (3,3) circle (0.25);
            \draw (3,2) circle (0.25);
            \draw (3,3.25) -- (3,1.75);
    
            \draw (3.75,1.75) rectangle (4.25,2.25);
            \draw (3.75,1.75) -- (4.25,2.25);
            \draw (3.75,2.25) -- (4.25,1.75);
    
            \draw (5,4) circle (0.25);
            \draw (5,3) circle (0.25);
            \draw (5,4.25) -- (5,2.75);
        \end{tikzpicture}
    \end{minipage}
    \caption{Pseudocode and network diagram representations of Dekker's \textsf{add2} algorithm. Note that the intermediate variables $s_0, s_1, t, u$ are unnamed in the network diagram representation, implicitly represented by the wire segments in between $\TwoSum$ gates.}
    \label{fig:add2}
\end{figure}

The purpose of an FPAN is to compute a nonoverlapping floating-point expansion of the exact value of the sum of its inputs. By the defining property of the $\TwoSum$ algorithm, this value is invariant to the application of a $\TwoSum$ gate to any two wires; it is only ever changed by discarding a wire. Therefore, to prove the correctness of an FPAN, it suffices to show the following:
\begin{enumerate}
    \item The non-discarded outputs, denoted by $(a_1, a_2, \dots, a_{n-k})$, must satisfy a \textit{nonoverlapping invariant}, such as $|a_{i+1}| \le \frac{1}{2} \ulp(a_i)$ for all $i = 1, \dots, n-k-1$.
    \item The discarded outputs, denoted by $(b_1, b_2, \dots, b_k)$, must satisfy an \textit{error bound}, such as $|b_i| \le \frac{1}{2} \ulp(a_{n-k})$ for all $i = 1, \dots, k$.
\end{enumerate}
These correctness properties are based on the notion of nonoverlapping defined by Equation~\eqref{eq:QDNO}, which is used in several notable software libraries, such as Bailey's QD library~\cite{QD}. However, a variety of nonoverlapping invariants and error bounds have been adopted by floating-point researchers in different situations. We briefly survey these notions below; see \cite[Sec.\ 14.2]{FPHandbook} for more details.

\begin{definition} \label{def:dominates}
    Let $x$ and $y$ be precision-$p$ floating-point numbers. Let $e_x$ and $e_y$ denote the exponents of $x$ and $y$, respectively, and let $\mathsf{ntz}_x$ denote the number of trailing zeros in the mantissa of $x$.
    \begin{itemize}
        \item We say $x$ \textbf{$\mathcal{S}$-dominates} $y$, denoted by $x \succ_{\mathcal{S}} y$, if $y = \pm\mathtt{0.0}$ or $x$ and $y$ are both nonzero and $e_x - e_y \ge p - \mathsf{ntz}_x$.
        \item We say $x$ \textbf{$\mathcal{P}$-dominates} $y$, denoted by $x \succ_{\mathcal{P}} y$, if $y = \pm\mathtt{0.0}$ or $x$ and $y$ are both nonzero and $e_x - e_y \ge p$.
        \item We say $x$ \textbf{$\ulp$-dominates} $y$, denoted by $x \succ_{\ulp} y$, if $|y| \le \ulp(x)$.
        \item We say $x$ \textbf{$\mathsf{QD}$-dominates} $y$, denoted by $x \succ_{\mathsf{QD}} y$, if $|y| \le \frac{1}{2}\ulp(x)$.
    \end{itemize}
\end{definition}
\begin{definition} \label{def:nonoverlapping}
    A floating-point expansion $(x_0, x_1, \dots, x_{n-1})$ is \textbf{$\mathcal{S}$-nonoverlapping} if $x_{k-1} \succ_{\mathcal{S}} x_k$ for all $k = 1, \dots, n-1$. We similarly define \textbf{$\mathcal{P}$-nonoverlapping}, \textbf{$\ulp$-nonoverlapping}, and \textbf{$\mathsf{QD}$-nonoverlapping} expansions with $\succ_{\mathcal{P}}$, $\succ_{\ulp}$, or $\succ_{\mathsf{QD}}$ replacing $\succ_{\mathcal{S}}$ in the preceding definition, respectively.
\end{definition}
We will return to this topic in Section~\ref{sec:Abstraction}, where we present a general abstraction that subsumes nonoverlapping invariants and error bounds in many forms.

Dekker's \textsf{add2} algorithm is notable for having very weak correctness guarantees. Assuming that the inputs $(x_0, x_1)$ and $(y_0, y_1)$ are $\mathcal{P}$-nonoverlapping, the sum $(z_0, z_1)$ computed by \textsf{add2} satisfies the following relative error bound:
\begin{equation}
    \frac{|(z_0 + z_1) - (x_0 + x_1 + y_0 + y_1)|}{|x_0 + x_1 + y_0 + y_1|} \le 4\u^2 \frac{|x_0 + x_1| + |y_0 + y_1|}{|x_0 + x_1 + y_0 + y_1|}
\end{equation}
Although this error bound is reasonably tight when $(x_0, x_1)$ and $(y_0, y_1)$ have the same sign, it can be extremely loose when $(x_0, x_1)$ and $(y_0, y_1)$ have different signs, in which case $|x_0 + x_1| + |y_0 + y_1|$ can be orders of magnitude larger than $|x_0 + x_1 + y_0 + y_1|$. Joldes, Muller, and Popescu~\cite{Joldes2017} identified example inputs for which \textsf{add2} computes sums with 100\% relative error, i.e., zero accurate bits compared to the true value of $x_0 + x_1 + y_0 + y_1$.

This observation highlights the surprising difficulty of computing accurate floating-point sums, even for as few as four inputs. At first glance, the network diagram shown in Figure~\ref{fig:add2} may not appear to have any obvious deficiencies. Indeed, when interpreted as a sorting network, this diagram gives a correct algorithm for partially sorting four inputs satisfying the preconditions $x_0 > x_1$ and $y_0 > y_1$. However, there are two key differences that make floating-point accumulation harder than sorting. First, the outputs of an FPAN not only need to be sorted by magnitude; they also require a degree of mutual separation to satisfy nonoverlapping invariants and error bounds. Second, unlike a comparator which merely reorders its inputs, a $\TwoSum$ gate actually modifies its inputs, potentially introducing new overlap and ordering issues with every operation.


\paragraph{Example: KBN Summation.} Kahan--Babu\v{s}ka--Neumaier (KBN) summation is an algorithm that uses $\TwoSum$ to compute floating-point sums with a running compensation term to improve the accuracy of the final result~\cite{Kahan1965,Babuska1972,Neumaier1974}. This technique is frequently used in floating-point programs and is implemented in both the Python and Julia standard libraries. In particular, Python's built-in \texttt{sum()} function uses KBN summation when given floating-point inputs~\cite{PythonDocs}.

In our graphical notation, the KBN algorithm is written as an FPAN with a double staircase structure illustrated in Figure~\ref{fig:KBN}. The first staircase computes the na\"ive floating-point sum of the inputs, while the second staircase computes the running compensation term used to correct the na\"ive sum.

\begin{figure}[t]
    \centering
    \begin{tikzpicture}[scale=0.5]
        \draw (0,5) node [anchor=east] {$x_5$} -- (13,5) node [anchor=west] {$s$};
        \draw (0,4) node [anchor=east] {$x_4$} -- (11.75,4);
        \draw (0,3) node [anchor=east] {$x_3$} -- (9.75,3);
        \draw (0,2) node [anchor=east] {$x_2$} -- (7.75,2);
        \draw (0,1) node [anchor=east] {$x_1$} -- (5.75,1);

        \draw (1,2) circle (0.25);
        \draw (1,1) circle (0.25);
        \draw (1,2.25) -- (1,0.75);

        \draw (3,3) circle (0.25);
        \draw (3,2) circle (0.25);
        \draw (3,3.25) -- (3,1.75);

        \draw (5,4) circle (0.25);
        \draw (5,3) circle (0.25);
        \draw (5,4.25) -- (5,2.75);

        \draw (5,2) circle (0.25);
        \draw (5,1) circle (0.25);
        \draw (5,2.25) -- (5,0.75);

        \draw (5.75,0.75) rectangle (6.25,1.25);
        \draw (5.75,0.75) -- (6.25,1.25);
        \draw (5.75,1.25) -- (6.25,0.75);

        \draw (7,5) circle (0.25);
        \draw (7,4) circle (0.25);
        \draw (7,5.25) -- (7,3.75);

        \draw (7,3) circle (0.25);
        \draw (7,2) circle (0.25);
        \draw (7,3.25) -- (7,1.75);

        \draw (7.75,1.75) rectangle (8.25,2.25);
        \draw (7.75,1.75) -- (8.25,2.25);
        \draw (7.75,2.25) -- (8.25,1.75);

        \draw (9,4) circle (0.25);
        \draw (9,3) circle (0.25);
        \draw (9,4.25) -- (9,2.75);

        \draw (9.75,2.75) rectangle (10.25,3.25);
        \draw (9.75,2.75) -- (10.25,3.25);
        \draw (9.75,3.25) -- (10.25,2.75);

        \draw (11,5) circle (0.25);
        \draw (11,4) circle (0.25);
        \draw (11,5.25) -- (11,3.75);

        \draw (11.75,3.75) rectangle (12.25,4.25);
        \draw (11.75,3.75) -- (12.25,4.25);
        \draw (11.75,4.25) -- (12.25,3.75);
    \end{tikzpicture}
    \caption{Network diagram for Kahan--Babu\v{s}ka--Neumaier summation applied to five inputs. This double staircase accumulation pattern generalizes to any number of inputs.}
    \label{fig:KBN}
\end{figure}
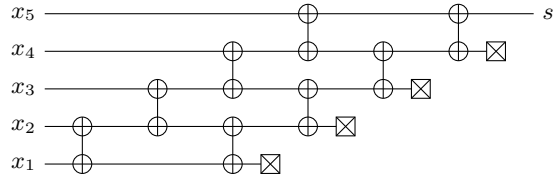



\section{Abstraction} \label{sec:Abstraction}

Existing methods for floating-point verification are ill-suited for FPANs. On one hand, techniques based on interval analysis or projection from real arithmetic are too imprecise to reason about error-free transformations. On the other hand, bit-blasting produces enormous satisfiability problems that are only tractable for tiny mantissa widths. Verifying FPANs requires a technique that can precisely specify the magnitude and shape of a floating-point sum without reasoning through the value of every bit. To this end, we introduce a new abstract domain for floating-point reasoning, which we call the SELTZO abstraction.

\begin{definition} \label{def:SELTZO}
Let $x$ be a floating-point number. The \textbf{sign-exponent leading-trailing zeros-ones (SELTZO) abstraction} of $x$ is the ordered 6-tuple $(s_x,\allowbreak e_x,\allowbreak \nlz_x,\allowbreak \nlo_x,\allowbreak \ntz_x,\allowbreak \nto_x)$ consisting of:
\begin{enumerate}
    \item the sign bit $s_x$ and exponent $e_x$ of $x$;
    \item the counts $(\nlz_x, \nlo_x)$ of leading zeros and ones, respectively, in the mantissa of $x$, ignoring the implicit leading bit; and
    \item the counts $(\ntz_x, \nto_x)$ of trailing zeros and ones, respectively, in the mantissa of $x$, ignoring the implicit leading bit.
\end{enumerate}
\end{definition}
For example, the SELTZO abstraction of $-1.0010011111_2 \times 2^7$ is $(1, 7, 2, 0, 0, 5)$, and the SELTZO abstraction of $+1.1111111111_2 \times 2^{-2}$ is $(0, -2, 0, 10, 0, 10)$. (Recall that the implicit leading bit is ignored when computing $\nlz_x$ and $\nlo_x$.)
When $x = \pm\mathtt{0.0}$, we set $e_x \coloneqq \emin - 1$ for consistency with IEEE 754 representation. Because we assume all floating-point numbers to be normalized or zero, $x = \pm\mathtt{0.0}$ is the only value we consider to have the property $e_x < \emin$.

Our definition of the SELTZO abstraction is motivated by several design considerations that we will explore in the remainder of this section.

\begin{itemize}
    \item The SELTZO abstraction precisely captures all of the nonoverlapping invariants specified in Definitions \ref{def:dominates} and \ref{def:nonoverlapping} using linear inequalities in the precision $p$ and the variables $(s_x, e_x, \nlz_x, \nlo_x, \ntz_x, \nto_x)$.
    \item The variables of the SELTZO abstraction are particularly well-behaved under the floating-point sum and $\TwoSum$ operations, which interact in a highly predictable fashion with long stretches of leading or trailing zeros and ones.
    \item As shown by the worst-case inputs found by Joldes, Muller, and Popescu~\cite{Joldes2017}, FPANs tend to exhibit pathological behavior near powers of two, which mark the boundaries between different exponent values. These are precisely the floating-point numbers with many leading mantissa zeros or ones.
\end{itemize}

Note that distinct SELTZO abstract values correspond to disjoint sets of concrete floating-point values. In particular, the bit counts $(\nlz_x, \nlo_x, \ntz_x, \nto_x)$ are \textit{not} cumulative; $\ntz_x = 3$ specifies concrete values with \textit{exactly} three trailing zeros, no more. As these bit counts increase, the number of concrete values that correspond to a particular abstract value drops exponentially, making the SELTZO abstraction more precise near exponent boundaries.

\subsection{SELTZO Correctness Properties}

To verify FPANs, the SELTZO abstraction must be able to express necessary and sufficient conditions for the correctness properties specified in Definitions \ref{def:dominates} and \ref{def:nonoverlapping}. The following proposition writes these properties as linear inequalities in the SELTZO variables $(s_x, e_x, \nlz_x, \nlo_x, \ntz_x, \nto_x)$ and the precision $p$.

\begin{proposition} \label{prop:nonoverlapping}
    Let $x$ and $y$ be precision-$p$ floating-point numbers. Using the notation of Definition~\ref{def:SELTZO}:  \\[-1.5\baselineskip]
    \begin{enumerate}
        \item $x \succ_{\mathcal{S}} y$ if and only if $y = \pm\mathtt{0.0}$ or $e_x - e_y \ge p - \ntz_x$.
        \item $x \succ_{\mathcal{P}} y$ if and only if $y = \pm\mathtt{0.0}$ or $e_x - e_y \ge p$.
        \item $x \succ_{\ulp} y$ if and only if $y = \pm\mathtt{0.0}$, $e_x - e_y > p - 1$, or $e_x - e_y = p - 1$ and $\ntz_y = p - 1$.
        \item $x \succ_{\mathsf{QD}} y$ if and only if $y = \pm\mathtt{0.0}$, $e_x - e_y > p$, or $e_x - e_y = p$ and $\ntz_y = p - 1$.
    \end{enumerate}
\end{proposition}
\begin{proof}
    Claims 1 and 2 follow immediately from the definitions of $\succ_{\mathcal{S}}$ and $\succ_{\mathcal{P}}$. To prove Claim 3, suppose $x$ and $y$ are both nonzero. Then $\ulp(x) = 2^{e_x - (p-1)}$, and the condition $|y| \le \ulp(x)$ is satisfied if and only if $|y| < 2^{e_x - (p-1)}$ or $|y| = 2^{e_x - (p-1)}$. The former is equivalent to $e_y < e_x - (p - 1)$, while the latter is equivalent to $e_y = e_x - (p - 1)$ and $\ntz_y = p - 1$ (i.e., all mantissa bits of $y$ are zero, excluding the implicit leading bit). Claim 4 is proven by an analogous argument with $\ulp(x)$ replaced by $\frac{1}{2}\ulp(x) = 2^{e_x - p}$. \qed
\end{proof}

\subsection{SELTZO Abstraction of $\TwoSum$}

In addition to expressing correctness properties, the SELTZO abstraction must also precisely capture the behavior of the $\TwoSum$ operation. In particular, it must be able to state relations between abstract inputs and outputs that rule out impossible input-output pairs. Relations of this type allow us to verify FPANs by ruling out the possibility of an invalid output.

The following propositions constrain the possible outputs of the $\TwoSum$ operation using linear inequalities in the SELTZO variables.

\begin{proposition} \label{prop:identity}
    Let $s$ and $e$ be floating-point numbers in the \textsf{binary16}, \textsf{bfloat16}, \textsf{binary32}, \textsf{binary64}, or \textsf{binary128} formats (defined in Table~\ref{tab:ieeeformats}). Using the notation of Definition~\ref{def:SELTZO}, $(s, e)$ is a fixed point of $\TwoSum$ if and only if at least one of the following conditions holds:
    \begin{enumerate}
        \item $e = \pm\mathtt{0.0}$.
        \item $e_s - e_e > p + 1$.
        \item $e_s - e_e = p + 1$ and one or more of the following sub-conditions holds:
        \begin{enumerate}
            \item $s_s = s_e$.
            \item $\mathsf{ntz}_s < p - 1$ (i.e., $s$ is not a power of $2$).
            \item $\mathsf{ntz}_e = p - 1$ (i.e., $e$ is a power of $2$).
        \end{enumerate}
        \item $e_s - e_e = p$, $\mathsf{ntz}_e = p - 1$, $\mathsf{ntz}_s \ge 1$, and one or more of the following sub-conditions holds:
        \begin{enumerate}
            \item $s_s = s_e$.
            \item $\mathsf{ntz}_s < p - 1$.
        \end{enumerate}
    \end{enumerate}
\end{proposition}
We prove Proposition~\ref{prop:identity} by direct verification with a floating-point SMT solver.

\begin{proposition} \label{prop:outputs}
    A pair of floating-point numbers $(s, e)$ is a possible output of $\TwoSum$ if and only if it satisfies the hypotheses of Proposition~\ref{prop:identity}.
\end{proposition}
\begin{proof}
    It suffices to show that every output of $\TwoSum$ is a fixed point, i.e., $\TwoSum$ is an idempotent operation. Given any floating-point numbers $(x, y)$, set $(s_1, e_1) \coloneqq \TwoSum(x, y)$ and $(s_2, e_2) \coloneqq \TwoSum(s_1, e_1)$. By the defining property of $\TwoSum$, we have $x + y = s_1 + e_1 = s_2 + e_2$, from which we deduce $s_2 = \RNE(s_1 + e_1) = \RNE(x + y) = s_1$ and $e_2 = x + y - s_2 = x + y - s_1 = e_1$. \qed
\end{proof}

Propositions \ref{prop:identity} and \ref{prop:outputs} express general constraints on the possible outputs of the $\TwoSum$ operation. However, verifying FPANs of practical interest requires stronger results formulated for specific classes of floating-point inputs. In the \hyperref[sec:Appendix]{Appendix} of this paper, we state a collection of over 70 lemmas that precisely characterize possible outputs for various input classes that cover the space of all possible inputs. In many cases, these lemmas are the strongest possible, in the sense that every abstract output not explicitly ruled out by the lemma is witnessed by some pair of concrete inputs. We verify these lemmas for the \textsf{binary16}, \textsf{bfloat16}, \textsf{binary32}, \textsf{binary64}, and \textsf{binary128} formats using a portfolio of SMT solvers, including Z3~\cite{Z3}, CVC5~\cite{CVC5}, MathSAT~\cite{MathSAT}, and Bitwuzla~\cite{BitwuzlaTool}.

\section{Verification} \label{sec:Verification}

With the SELTZO abstraction defined, we can now state our automatic FPAN verification procedure. Suppose we are given an FPAN $F$ and a property $P$ that is expressible as a logical combination of linear inequalities in the SELTZO varibles $(s_x, e_x, \nlz_x, \nlo_x, \ntz_x, \nto_x)$. We construct a statement $S_{P,F}$ in quantifier-free Presburger arithmetic that expresses the existence of an abstract counterexample to $P(F)$. If this statement is unsatisfiable, then $P(F)$ has no abstract counterexamples. Hence, no concrete counterexample exists, and $P(F)$ is proven.

The first step of this procedure is to assign a unique label $v_i$ to every wire segment in $F$. Every $\TwoSum$ gate delineates a new segment of the wires it connects. Thus, an FPAN with $n$ wires and $g$ gates has $n + 2g$ distinct wire segments. We then introduce SELTZO variables $(s_{v_i}, e_{v_i}, \nlz_{v_i}, \nlo_{v_i}, \ntz_{v_i}, \nto_{v_i})$ indexed by the labels $v_i$, creating a total of $6n + 12g$ variables.

We construct the statement $S_{P,F}$ as a logical conjunction of three types of conditions: (1) \textit{consistency conditions} that enforce the validity of the abstract values $(s_{v}, e_{v}, \nlz_{v}, \nlo_{v}, \ntz_{v}, \nto_{v})$; (2) \textit{execution conditions} that constrain the possible outputs of each $\TwoSum$ gate; and (3) \textit{counterexample conditions} that encode the negation of the property $P(F)$ that we wish to prove.

We first state the consistency conditions in Equations~\eqref{eq:consistencyfirst}--\eqref{eq:consistencylast}, which give a necessary and sufficient characterization of all valid SELTZO abstract values. One copy of these consistency conditions is instantiated for each label $v_i$.
\begin{enumerate}
    \item The sign bit must be zero or one, and the exponent must be bounded below.
    \begin{equation} \label{eq:consistencyfirst}
        (s_v = 0) \vee (s_v = 1) \qquad\qquad e_v \ge \emin - 1
    \end{equation}
    \item If a floating-point variable is zero (i.e., $e_v = \emin - 1$), then its mantissa must consist entirely of zeros.
    \begin{equation}
        (e_v = \emin - 1) \to [(\nlz_v = \ntz_v = p - 1) \wedge (\nlo_v = \nto_v = 0)]
    \end{equation}
    \item The leading and trailing bits of the mantissa are either 0 or 1.
    \begin{align}
        [(\nlz_v > 0) \wedge (\nlo_v = 0)] &\vee [(\nlz_v = 0) \wedge (\nlo_v > 0)] \\
        [(\ntz_v > 0) \wedge (\nto_v = 0)] &\vee [(\ntz_v = 0) \wedge (\nto_v > 0)]
    \end{align}
    \item The number of leading and trailing bits must be bounded by $p-1$, the width of the mantissa.
    \begin{align}
        (\nlz_v = \ntz_v = p - 1) &\vee (\nlz_v + \ntz_v < p - 1) \\
        (\nlo_v = \nto_v = p - 1) &\vee (\nlo_v + \nto_v < p - 1) \\
        (\nlz_v + \nto_v = p - 1) &\vee (\nlz_v + \nto_v < p - 2) \\
        (\ntz_v + \nlo_v = p - 1) &\vee (\ntz_v + \nlo_v < p - 2) \label{eq:consistencylast}
    \end{align}
\end{enumerate}
The upper bound of $p-2$ in the last two conditions expresses the constraint that, in a bit string of the form $00 \cdots 0 b 11 \cdots 1$, the middle bit $b$ either belongs to the group of leading zeros or trailing ones. Thus, $\nlz_v + \nto_v \ne p - 2$.

We then adjoin a set of execution conditions to $S_{P,F}$ for each $\TwoSum$ gate in $F$. Each of these sets comprises hundreds of inequalities with complex logical interdependencies, collectively stated in Propositions \ref{prop:nonoverlapping}, \ref{prop:identity}, \ref{prop:outputs}, and the lemmas given in the \hyperref[sec:Appendix]{Appendix}.

Finally, we add counterexample conditions to encode the \textit{negation} of the property $P(F)$ that we wish to prove. These conditions must be formulated such that any concrete counterexample to the desired property $P(F)$ has SELTZO abstract values that violate the counterexample conditions. Any property that is expressible as a logical combination of linear inequalities in the SELTZO variables can be used to construct the counterexample conditions. By Proposition~\ref{prop:nonoverlapping}, this class includes the following nonoverlapping invariants:
\begin{itemize}
    \item $v \succ_{\mathcal{S}} w$ is expressed as $(e_w = \emin - 1) \vee (e_v - e_w \ge p - \ntz_v)$.
    \item $v \succ_{\mathcal{P}} w$ is expressed as $(e_w = \emin - 1) \vee (e_v - e_w \ge p)$.
    \item $v \succ_{\ulp} w$ is expressed as \\ $(e_w = \emin - 1) \vee (e_v - e_w > p - 1) \vee [(e_v - e_w = p - 1) \wedge (\ntz_w = p - 1)]$.
    \item $v \succ_{\mathsf{QD}} w$ is expressed as \\ $(e_w = \emin - 1) \vee (e_v - e_w > p) \vee [(e_v - e_w = p) \wedge (\ntz_w = p - 1)]$.
\end{itemize}
The counterexample conditions can also include inequalities between floating-point values of the form $|w| < 2^{-k} |v|$, which are particularly useful for formulating relative error bounds. As shown by the following proposition, these can be expressed as $(e_w = \emin - 1) \vee (e_v - e_w > k)$.

\begin{proposition} \label{prop:relerr}
    Let $x$ and $y$ be nonzero floating-point numbers. If $e_x - e_y > k$, then $|y| < 2^{-k} |x|$.
\end{proposition}
\begin{proof}
    By definition, $|x| \in [2^{e_x}, (2 - 2\u) 2^{e_x}]$ and $|y| \in [2^{e_y}, (2 - 2\u) 2^{e_y}]$. Taking the lower bound for $x$ and the upper bound for $|y|$, we have:
    \begin{equation}
        \frac{|y|}{|x|} \le \frac{(2 - 2\u) 2^{e_y}}{2^{e_x}} = (1 - \u) 2^{-(e_x - e_y - 1)} \le (1 - \u) 2^{-k} < 2^{-k}
    \end{equation}
    This proves $|y| < 2^{-k} |x|$, as required. \qed
\end{proof}

\subsection{Handling Subnormal Values} \label{sec:Subnormal}

Note that our consistency conditions do not impose an upper bound on the exponent $e_x$ of any abstract value. This means that the unsatisfiability of $S_{P,F}$ actually proves a stronger result: there are no counterexamples to $P(F)$ even in a hypothetical floating-point format with infinite exponent range. This key observation enables all of the analysis we have presented so far to generalize to subnormal inputs, as long as the property $P$ depends only on differences between exponents $e_x - e_y$ and not absolute exponent values.

\begin{proposition}
    If $S_{P,F}$ is unsatisfiable and the property $P$ is expressible using only $(s_{v}, \nlz_{v}, \nlo_{v}, \ntz_{v}, \nto_{v})$ and exponent differences of the form $e_v - e_w$, then $P(F)$ holds for all finite inputs.
\end{proposition}
\begin{proof}
    We prove the contrapositive. Suppose there exists a concrete counterexample to $P(F)$ consisting of finite (possibly subnormal) floating-point input values $(x_1, x_2, \dots, x_n)$. Any property $P$ expressible in this form is invariant under a global shift of all exponents. Hence, for any sufficiently large $k$, the SELTZO abstraction of $(2^k x_1, 2^k x_2, \dots, 2^k x_n)$ yields an abstract counterexample to $P(F)$ consisting only of normalized abstract values. We may ignore the possibility of overflow in this construction because our abstraction uses an unbounded exponent range. This proves that $S_{P,F}$ is satisfiable. \qed
\end{proof}

\section{Results} \label{sec:Results}

We apply the machinery developed in this paper to prove tight error bounds for both the best known double-double addition algorithm (\textsf{ddadd}), widely used in software libraries~\cite{XBLAS,CRLibM,COREMATH}, and a novel algorithm that is simultaneously faster and more accurate than \textsf{ddadd}. Our new algorithm, named \textsf{madd} (for ``More Accurate Double-Double addition''), reduces the relative error of double-double addition from $3\u^2$ to $2\u^2$ while lowering circuit depth from 5 to 4. FPANs for both algorithms are given in Figure~\ref{fig:networks}.

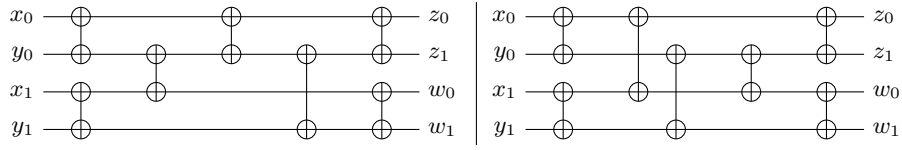
\begin{figure}[t]
    \begin{minipage}{0.52\textwidth}
        \centering
        \begin{tikzpicture}[scale=0.5]
            \draw (0,4) node [anchor=east] {$x_0$} -- (10,4) node [anchor=west] {$z_0$};
            \draw (0,3) node [anchor=east] {$y_0$} -- (10,3) node [anchor=west] {$z_1$};
            \draw (0,2) node [anchor=east] {$x_1$} -- (10,2) node [anchor=west] {$w_0$};
            \draw (0,1) node [anchor=east] {$y_1$} -- (10,1) node [anchor=west] {$w_1$};
    
            \draw (1,4) circle (0.25);
            \draw (1,3) circle (0.25);
            \draw (1,4.25) -- (1,2.75);
    
            \draw (1,2) circle (0.25);
            \draw (1,1) circle (0.25);
            \draw (1,2.25) -- (1,0.75);
    
            \draw (3,3) circle (0.25);
            \draw (3,2) circle (0.25);
            \draw (3,3.25) -- (3,1.75);
    
            \draw (5,4) circle (0.25);
            \draw (5,3) circle (0.25);
            \draw (5,4.25) -- (5,2.75);
    
            \draw (7,3) circle (0.25);
            \draw (7,1) circle (0.25);
            \draw (7,3.25) -- (7,0.75);
    
            \draw (9,4) circle (0.25);
            \draw (9,3) circle (0.25);
            \draw (9,4.25) -- (9,2.75);
    
            \draw (9,2) circle (0.25);
            \draw (9,1) circle (0.25);
            \draw (9,2.25) -- (9,0.75);
        \end{tikzpicture}
    \end{minipage}\hfill\vline\hfill\begin{minipage}{0.47\textwidth}
        \centering
        \begin{tikzpicture}[scale=0.5]
            \draw (0,4) node [anchor=east] {$x_0$} -- (9,4) node [anchor=west] {$z_0$};
            \draw (0,3) node [anchor=east] {$y_0$} -- (9,3) node [anchor=west] {$z_1$};
            \draw (0,2) node [anchor=east] {$x_1$} -- (9,2) node [anchor=west] {$w_0$};
            \draw (0,1) node [anchor=east] {$y_1$} -- (9,1) node [anchor=west] {$w_1$};
    
            \draw (1,4) circle (0.25);
            \draw (1,3) circle (0.25);
            \draw (1,4.25) -- (1,2.75);
    
            \draw (1,2) circle (0.25);
            \draw (1,1) circle (0.25);
            \draw (1,2.25) -- (1,0.75);
    
            \draw (3,4) circle (0.25);
            \draw (3,2) circle (0.25);
            \draw (3,4.25) -- (3,1.75);
    
            \draw (4,3) circle (0.25);
            \draw (4,1) circle (0.25);
            \draw (4,3.25) -- (4,0.75);
    
            \draw (6,3) circle (0.25);
            \draw (6,2) circle (0.25);
            \draw (6,3.25) -- (6,1.75);
    
            \draw (8,4) circle (0.25);
            \draw (8,3) circle (0.25);
            \draw (8,4.25) -- (8,2.75);
    
            \draw (8,2) circle (0.25);
            \draw (8,1) circle (0.25);
            \draw (8,2.25) -- (8,0.75);
        \end{tikzpicture}
    \end{minipage}
    \caption{Augmented network diagrams for \textsf{ddadd} (left) and \textsf{madd} (right), our new double-double addition algorithm, with error terms explicitly computed and named. The extra $\TwoSum$ gate used to compute these error terms serves only to facilitate our analysis and should not be included in an actual implementation of either algorithm.}
    \label{fig:augmented}
\end{figure}

\begin{theorem} \label{thm:ddadd}
    Let $(x_0, x_1)$ and $(y_0, y_1)$ be floating-point expansions in the \textsf{binary16}, \textsf{bfloat16}, \textsf{binary32}, \textsf{binary64}, or \textsf{binary128} format with $x_0 = \RNE(x_0 + x_1)$ and $y_0 = \RNE(y_0 + y_1)$. The \textsf{ddadd} algorithm, depicted in Figure~\ref{fig:networks} (left), computes a floating-point expansion $(z_0, z_1)$ that approximates the exact sum $x_0 + x_1 + y_0 + y_1$ with relative error bounded above by $(1 + 2\u) 2^{-(2p-2)} \approx 4\u^2$.
\end{theorem}

\begin{theorem} \label{thm:madd}
    Let $(x_0, x_1)$ and $(y_0, y_1)$ be floating-point expansions in the \textsf{binary16}, \textsf{bfloat16}, \textsf{binary32}, \textsf{binary64}, or \textsf{binary128} format with $x_0 = \RNE(x_0 + x_1)$ and $y_0 = \RNE(y_0 + y_1)$. The \textsf{madd} algorithm, depicted in Figure~\ref{fig:networks} (right), computes a floating-point expansion $(z_0, z_1)$ that approximates the exact sum $x_0 + x_1 + y_0 + y_1$ with relative error bounded above by $(1 + 2\u) 2^{-(2p-1)} \approx 2\u^2$.
\end{theorem}

We prove Theorems~\ref{thm:ddadd} and \ref{thm:madd} by using the SELTZO abstraction to establish $P(F)$ for a suitably chosen property $P$ given below. Some algebraic manipulation is necessary to adapt $P$ into a true relative error bound that considers all outputs.

\begin{proof}
    Consider the augmented FPANs depicted in Figure~\ref{fig:augmented}, which include an extra $\TwoSum$ gate to compute the error terms $(w_0, w_1)$. Using the SELTZO encodings stated in Propositions~\ref{prop:identity} and \ref{prop:relerr}, we use an SMT solver to prove
    \begin{equation}
        P \coloneqq (x_0 = \RNE(x_0 + x_1)) \wedge (y_0 = \RNE(y_0 + y_1)) \to (|w_0| < 2^{-k} |z_0|)
    \end{equation}
    where $k = 2p - 2$ for \textsf{ddadd} and $k = 2p - 1$ for \textsf{madd}. Note that $x_0 = \RNE(x_0 + x_1)$ is equivalent to the statement that $(x_0, x_1)$ is a fixed point of $\TwoSum$. Moreover, by Proposition~\ref{prop:outputs}, we have $z_0 = \RNE(z_0, z_1)$ and $w_0 = \RNE(w_0, w_1)$ because $(z_0, z_1)$ and $(w_0, w_1)$ are $\TwoSum$ outputs. It follows that $w_0 + w_1 = (1 + \delta_w) w_0$ and $z_0 + z_1 = (1 + \delta_z) z_0$ for some $|\delta_w|, |\delta_z| \le \u$. Hence, the relative error of these FPANs is bounded above by:
    \begin{multline}
        \frac{|(z_0 + z_1) - (x_0 + x_1 + y_0 + y_1)|}{|x_0 + x_1 + y_0 + y_1|}
        = \frac{|(1 + \delta_w) w_0|}{|(1 + \delta_z) z_0 + (1 + \delta_w) w_0|} \\
        \le \frac{(1 + \u) |w_0|}{(1 - \u)(|z_0| - |w_0|)}
        \le \frac{(1 + \u) 2^{-k}}{1 - (1 - \u)2^{-k}}
        \le (1 + 2\u) 2^{-k}
    \end{multline}
    This expression is asymptotically equivalent to $4\u^2 + O(\u^3)$ for \textsf{ddadd} and $2\u^2 + O(\u^3)$ for \textsf{madd}. \qed
\end{proof}

In Table~\ref{tab:verificationtime}, we compare the speed of directly verifying the property $P$ defined above using a variety of floating-point theory solvers (\texttt{QF{\textunderscore}FP}) to resolving the satisfiability problems $S_{P,F}$ constructed by the SELTZO abstraction with a Presburger arithmetic solver (\texttt{QF{\textunderscore}LIA}). We benchmark a portfolio of state-of-the-art SMT solvers using the latest software versions available at the time of writing. In all cases, the SELTZO abstraction produces problems that are many orders of magnitude faster to solve. We also note that our SELTZO solve times remain constant with respect to the precision $p$ of the floating-point format, enabling scalability to wide formats that are intractable for bit-blasting.

\begin{table}[t]
    \centering
    \setlength\tabcolsep{3pt}
    \begin{tabular}{|c|c||c|c|c|c|c||c|} \hline
        FPAN           & Format             & Z3  & CVC5     & MathSAT  & Bitwuzla & Colibri2 & SELTZO    \\\hline\hline
        \textsf{ddadd} &  \textsf{binary16} & DNF & 153 min  & DNF      & 72 min   & N/A      & 0.927 sec \\\hline
        \textsf{madd}  &  \textsf{binary16} & DNF & 120 min  & 3898 min & 72 min   & N/A      & 0.713 sec \\\hline
        \textsf{ddadd} &  \textsf{bfloat16} & DNF & 704 min  & DNF      & 71 min   & N/A      & 0.838 sec \\\hline
        \textsf{madd}  &  \textsf{bfloat16} & DNF & 946 min  & DNF      & 99 min   & N/A      & 0.689 sec \\\hline
        \textsf{ddadd} &  \textsf{binary32} & DNF & 1088 min & DNF      & 640 min  & N/A      & 0.774 sec \\\hline
        \textsf{madd}  &  \textsf{binary32} & DNF & 1019 min & DNF      & 518 min  & N/A      & 0.722 sec \\\hline
        \textsf{ddadd} &  \textsf{binary64} & DNF & DNF      & DNF      & DNF      & N/A      & 0.623 sec \\\hline
        \textsf{madd}  &  \textsf{binary64} & DNF & DNF      & DNF      & DNF      & N/A      & 0.923 sec \\\hline
        \textsf{ddadd} & \textsf{binary128} & DNF & DNF      & DNF      & DNF      & N/A      & 0.880 sec \\\hline
        \textsf{madd}  & \textsf{binary128} & DNF & DNF      & DNF      & DNF      & N/A      & 0.991 sec \\\hline
    \end{tabular} \\[\baselineskip]
    \caption{Execution time for various SMT solvers to directly verify property $P$ in a floating-point theory (\texttt{QF{\textunderscore}FP}) compared to deciding the satisfiability of $S_{P,F}$ in quantifier-free Presburger arithmetic (\texttt{QF{\textunderscore}LIA}). A ``DNF'' entry indicates that a solver did not terminate within three days, while an ``N/A'' entry indicates that a solver rejected the problem as unsolvable. These benchmarks were performed on an AMD Ryzen 9 9950X processor using Z3 4.13.4, CVC5 1.2.0, MathSAT 5.6.11, Bitwuzla 0.7.0, and Colibri2 0.4. SELTZO satisfiability problems were solved using Z3 4.13.4.}
    \label{tab:verificationtime}
\end{table}

Note that our bound of $4\u^2$ for \textsf{ddadd} is slightly looser than the $3\u^2$ bound proven by Joldes, Muller, and Popescu~\cite{Joldes2017}. Our use of the SELTZO abstraction restricts us to proving bounds of the form $2^k \u^n + O(\u^{n+1})$, which we say are \textit{tight to the nearest bit} if the leading constant factor is the smallest possible power of two. The following example due to Muller and Rideau~\cite{Muller2022} shows that our \textsf{ddadd} bound is indeed tight to the nearest bit:
\begin{equation}
    x_0 \coloneqq 1 \qquad
    x_1 \coloneqq \u - \u^2 \qquad
    y_0 \coloneqq -\frac{1}{2} + \frac{\u}{2} \qquad
    y_1 \coloneqq -\frac{\u^2}{2} + \u^3
\end{equation}
It is straightforward to verify that the result returned by \textsf{ddadd} has relative error $\approx 3\u^2$ on these inputs. Moreover, the following example shows that our \textsf{madd} bound is also tight to the nearest bit:
\begin{equation}
    x_0 \coloneqq 1 + 2\u \qquad
    x_1 \coloneqq -\frac{\u}{2} - 2\u^2 \qquad
    y_0 \coloneqq -\u \qquad
    y_1 \coloneqq -\frac{\u^2}{2} - \u^3
\end{equation}
The result returned by \textsf{madd} has relative error $\approx 1.5\u^2$ on these inputs.

\subsection{Ablated Abstractions}

It is natural to ask whether simpler abstractions than SELTZO can prove Theorems \ref{thm:ddadd} and \ref{thm:madd}. To investigate this question, we consider two ablated abstractions that use subsets of the SELTZO variables $(s_x,\allowbreak e_x,\allowbreak \nlz_x,\allowbreak \nlo_x,\allowbreak \ntz_x,\allowbreak \nto_x)$.
\begin{definition}
    Let $x$ be a floating-point number. Using the notation of Definition~\ref{def:SELTZO}, the \textbf{SE abstraction} of $x$ is the 2-tuple $(s_x, e_x)$, and the \textbf{SETZ abstraction} of $x$ is the 3-tuple $(s_x, e_x, \ntz_x)$.
\end{definition}
The SE and SETZ abstractions are natural simplifications of SELTZO that omit some or all information about mantissa bit patterns. Notably, the SETZ abstraction is still expressive enough to capture all of the nonoverlapping invariants stated in Proposition~\ref{prop:nonoverlapping}, none of which refer to leading bits or trailing ones.

To perform FPAN verification with these simpler abstractions, we construct SE and SETZ satisfiability problems $S_{P,F}$ using special sets of reduced lemmas given in the \hyperref[sec:Appendix]{Appendix} that make no reference to the omitted variables $\nlz_x$, $\nlo_x$, $\nto_x$, and (for the SE abstraction) $\ntz_x$. We then find the maximal value of $k$ for which the property $P$ defined above holds in each abstraction. These values are reported in Table~\ref{tab:ablation}. While SE and SETZ are unable to match the tight error bounds provable in the SELTZO abstraction, they still establish nontrivial bounds that are difficult to prove by hand. However, these results show that modeling leading mantissa bits is necessary to prove tight FPAN error bounds.

\begin{table}[t]
    \centering
    \setlength\tabcolsep{6pt}
    \begin{tabular}{|c||c|c|c|} \hline
        FPAN & SE & SETZ & SELTZO \\\hline\hline
         \textsf{ddadd} & $2^{-(2p - 7)} = 128\u^2$ & $2^{-(2p - 4)} = 16\u^2$ & $2^{-(2p - 2)} = 4\u^2$ \\\hline
         \textsf{madd}  & $2^{-(2p - 6)} = 64\u^2$ & $2^{-(2p - 3)} = 8\u^2$ & $2^{-(2p - 1)} = 2\u^2$ \\\hline
    \end{tabular} \\[\baselineskip]
    \caption{Strongest relative error bounds for \textsf{ddadd} and \textsf{madd} that are provable in the SE, SETZ, and SELTZO abstractions.}
    \label{tab:ablation}
\end{table}

\section{Related Work} \label{sec:RelatedWork}

\paragraph{Automatic floating-point verification.}

Existing techniques for automatic verification of floating-point algorithms rely on abstractions derived from rational or real arithmetic, such as interval enclosures, polyhedra, and relational domains~\cite{Chen2008,Monniaux2008,Chapoutot2010,Rivera2024}. These abstractions are fundamentally inapplicable to FPANs. The notion of an error-free transformation like $\TwoSum$ only exists in finite-precision rounded arithmetic and has no semantically equivalent analogue in exact real arithmetic. Indeed, in real arithmetic, $\TwoSum(x, y) = (x + y, 0)$ is a trivial operation. As our results show, precise reasoning about FPANs requires explicitly modeling the interaction between the sum and error terms, taking into account the shape of the mantissa, which these abstractions cannot express.

\paragraph{Interactive floating-point verification.}

To the best of our knowledge, the only existing methods that can reason about error-free transformations and FPANs use interactive, as opposed to automatic, theorem provers. Tools such as Flocq~\cite{Flocq} and Gappa~\cite{Gappa} have been used to verify algorithms involving FPANs~\cite{CRLibM,COREMATH}, but they require a high degree of user expertise to construct sophisticated proof scripts. The SELTZO abstraction provides a complementary approach that could be integrated with these tools to provide a greater degree of automation.


\paragraph{Other approaches to high-precision arithmetic.}

Libraries for arbitrary-precision arithmetic, including GMP, MPFR, and Arb, make no internal use of floating-point operations~\cite{GMP,MPFR,Arb}. Instead, they implement arithmetic purely in terms of digit-by-digit integer operations. This approach allows for truly arbitrary precision, unconstrained by floating-point overflow and underflow limits, and avoids the complexity of propagating rounding errors that accompanies the use of error-free transformations. However, at moderate precision levels (2--8 machine words), these algorithms are many times slower than FPANs, requiring complex branching logic and more operations per bit on average. While FPANs are hard to discover and prove correct, they enable high-performance branch-free arithmetic that massively accelerates high-precision scientific applications.

\paragraph{Scalable abstraction in other domains.}

Recent work on the Bitwuzla SMT solver~\cite{BitwuzlaAbstraction} has used lemmas for integer multiplication and division to accelerate bit-vector verification, enabling scalability to previously intractable bit-widths. The SELTZO abstraction can be thought of as a floating-point analogue of this approach, characterizing the $\TwoSum$ operation in a precision-independent fashion to avoid full-width bit-blasting. One notable difference is that our approach does not require abstraction refinement tailored to a specific mantissa width or FPAN.

\paragraph{Sorting networks.}

FPANs are close analogues of sorting networks. Although they compute different operations, both are branch-free algorithms that accelerate sorting or accumulation of a fixed number of inputs in a data-parallel fashion. Because they both occupy a similar large combinatorial space, techniques for discovering and optimizing sorting networks, such as evolutionary algorithms~\cite{SorterHunter}, can also be applied to FPANs. Our work on the SELTZO abstraction provides an efficient correctness check that can be used to perform such a search.

\section{Conclusion} \label{sec:Conclusion}

In this paper, we defined \textit{floating-point accumulation networks} (FPANs) and showed that FPANs are key building blocks used to construct extended-precision floating-point algorithms. To address the difficulty of analyzing floating-point rounding errors, which critically affects the correctness of FPANs, we introduced the SELTZO abstraction to enable efficient and precise automatic reasoning about FPANs using SMT solvers. Using the SELTZO abstraction, we developed computer-verified proofs of a number of FPAN correctness properties. In particular, we automatically and rigorously proved that \textsf{madd}, a novel FPAN for double-double addition, is simultaneously faster and more accurate than \textsf{ddadd}, the previous best known algorithm.



\bibliographystyle{splncs04}
\bibliography{bibliography.bib}

\begin{thebibliography}{10}
\providecommand{\url}[1]{\texttt{#1}}
\providecommand{\urlprefix}{URL }
\providecommand{\doi}[1]{https://doi.org/#1}

\bibitem{WebAssembly}
{WebAssembly Core Specification}, \url{https://www.w3.org/TR/wasm-core-2/}

\bibitem{IEEE7541985}
{IEEE} standard for binary floating-point arithmetic. ANSI/IEEE Std 754-1985
  pp. 1--20 (1985). \doi{10.1109/IEEESTD.1985.82928}

\bibitem{IEEE7542008}
{IEEE} standard for floating-point arithmetic. IEEE Std 754-2008 pp. 1--70
  (2008). \doi{10.1109/IEEESTD.2008.4610935}

\bibitem{IEEE7542019}
{IEEE} standard for floating-point arithmetic. IEEE Std 754-2019 (Revision of
  IEEE 754-2008) pp. 1--84 (2019). \doi{10.1109/IEEESTD.2019.8766229}

\bibitem{Babuska1972}
Babu\v{s}ka, I.: Numerical stability in problems of linear algebra. SIAM
  Journal on Numerical Analysis  \textbf{9}(1),  53--77 (1972).
  \doi{10.1137/0709008}, \url{https://doi.org/10.1137/0709008}

\bibitem{Bailey1993}
Bailey, D.H., Krasny, R., Pelz, R.: Multiple precision, multiple processor
  vortex sheet roll-up computation. Society for Industrial and Applied
  Mathematics (SIAM), Philadelphia, PA (United States) (12 1993),
  \url{https://www.osti.gov/biblio/54379}

\bibitem{BaileySoftware}
Bailey, D.H.: High-precision software directory.
  \url{https://www.davidhbailey.com/dhbsoftware/} (2024)

\bibitem{Bailey2012}
Bailey, D.H., Borwein, J.M.: Hand-to-hand combat with thousand-digit integrals.
  Journal of Computational Science  \textbf{3}(3),  77--86 (2012).
  \doi{https://doi.org/10.1016/j.jocs.2010.12.004},
  \url{https://www.sciencedirect.com/science/article/pii/S1877750310000773},
  scientific Computation Methods and Applications

\bibitem{Bailey2005}
Bailey, D.: High-precision floating-point arithmetic in scientific computation.
  Computing in Science \& Engineering  \textbf{7}(3),  54--61 (2005).
  \doi{10.1109/MCSE.2005.52}

\bibitem{CVC5}
Barbosa, H., Barrett, C.W., Brain, M., Kremer, G., Lachnitt, H., Mann, M.,
  Mohamed, A., Mohamed, M., Niemetz, A., N{\"{o}}tzli, A., Ozdemir, A.,
  Preiner, M., Reynolds, A., Sheng, Y., Tinelli, C., Zohar, Y.: cvc5: {A}
  versatile and industrial-strength {SMT} solver. In: Fisman, D., Rosu, G.
  (eds.) Tools and Algorithms for the Construction and Analysis of Systems -
  28th International Conference, {TACAS} 2022, Held as Part of the European
  Joint Conferences on Theory and Practice of Software, {ETAPS} 2022, Munich,
  Germany, April 2-7, 2022, Proceedings, Part {I}. Lecture Notes in Computer
  Science, vol. 13243, pp. 415--442. Springer (2022).
  \doi{10.1007/978-3-030-99524-9\_24},
  \url{https://doi.org/10.1007/978-3-030-99524-9\_24}

\bibitem{Boldo2017Robustness}
Boldo, S., Graillat, S., Muller, J.M.: On the robustness of the 2sum and
  fast2sum algorithms. ACM Trans. Math. Softw.  \textbf{44}(1) (Jul 2017).
  \doi{10.1145/3054947}, \url{https://doi.org/10.1145/3054947}

\bibitem{Boldo2017Formal}
Boldo, S., Joldes, M., Muller, J.M., Popescu, V.: Formal verification of a
  floating-point expansion renormalization algorithm. In: Ayala-Rinc{\'o}n, M.,
  Mu{\~{n}}oz, C.A. (eds.) Interactive Theorem Proving. pp. 98--113. Springer
  International Publishing, Cham (2017)

\bibitem{Flocq}
Boldo, S., Melquiond, G.: Flocq: A unified library for proving floating-point
  algorithms in coq. In: 2011 IEEE 20th Symposium on Computer Arithmetic. pp.
  243--252 (2011). \doi{10.1109/ARITH.2011.40}

\bibitem{Boldo2011}
Boldo, S., Muller, J.M.: Exact and approximated error of the {FMA}. IEEE
  Transactions on Computers  \textbf{60}(2),  157--164 (2011).
  \doi{10.1109/TC.2010.139}

\bibitem{Chapoutot2010}
Chapoutot, A.: Interval {Slopes} as a {Numerical} {Abstract} {Domain} for
  {Floating}-{Point} {Variables}. In: Cousot, R., Martel, M. (eds.) Static
  {Analysis}. pp. 184--200. Springer, Berlin, Heidelberg (2010).
  \doi{10.1007/978-3-642-15769-1_12}

\bibitem{Chen2008}
Chen, L., Miné, A., Cousot, P.: A {Sound} {Floating}-{Point} {Polyhedra}
  {Abstract} {Domain}. In: Ramalingam, G. (ed.) Programming {Languages} and
  {Systems}. pp. 3--18. Springer, Berlin, Heidelberg (2008).
  \doi{10.1007/978-3-540-89330-1_2}

\bibitem{MathSAT}
Cimatti, A., Griggio, A., Schaafsma, B.J., Sebastiani, R.: The {MathSAT5} {SMT}
  {Solver}. In: Piterman, N., Smolka, S.A. (eds.) Tools and {Algorithms} for
  the {Construction} and {Analysis} of {Systems}. pp. 93--107. Springer,
  Berlin, Heidelberg (2013). \doi{10.1007/978-3-642-36742-7_7}

\bibitem{Collange2016}
Collange, C., Joldes, M., Muller, J.M., Popescu, V.: Parallel floating-point
  expansions for extended-precision {GPU} computations. In: 2016 IEEE 27th
  International Conference on Application-specific Systems, Architectures and
  Processors (ASAP). pp. 139--146 (2016). \doi{10.1109/ASAP.2016.7760783}

\bibitem{CRLibM}
Daramy-Loirat, C., Defour, D., de~Dinechin, F., Gallet, M., Gast, N., Lauter,
  C., Muller, J.M.: {CR-LIBM A library of correctly rounded elementary
  functions in double-precision}. Research report, {LIP,} (Dec 2006),
  \url{https://ens-lyon.hal.science/ensl-01529804}

\bibitem{Z3}
De~Moura, L., Bj\o{}rner, N.: Z3: an efficient smt solver. In: Proceedings of
  the Theory and Practice of Software, 14th International Conference on Tools
  and Algorithms for the Construction and Analysis of Systems. p. 337–340.
  TACAS'08/ETAPS'08, Springer-Verlag, Berlin, Heidelberg (2008)

\bibitem{Dekker1971}
Dekker, T.J.: A floating-point technique for extending the available precision.
  Numerische Mathematik  \textbf{18}(3),  224--242 (Jun 1971).
  \doi{10.1007/BF01397083}, \url{https://doi.org/10.1007/BF01397083}

\bibitem{Gappa}
de~Dinechin, F., Lauter, C., Melquiond, G.: {Certifying floating-point
  implementations using Gappa} (Dec 2007),
  \url{https://ens-lyon.hal.science/ensl-00200830}, working paper or preprint

\bibitem{SorterHunter}
Dobbelaere, B.: {SorterHunter} (2024),
  \url{https://github.com/bertdobbelaere/SorterHunter}

\bibitem{ECMAScript}
{ECMA International}: Standard ECMA-262 - ECMAScript Language Specification. 15
  edn. (2024),
  \url{https://ecma-international.org/publications-and-standards/standards/ecma-262/}

\bibitem{AccurateArithmetic}
Elrod, C., F{\'e}votte, F.: {Accurate and Efficiently Vectorized Sums and Dot
  Products in Julia} (Aug 2019), \url{https://hal.science/hal-02265534},
  version submitted to the Correctness2019 workshop

\bibitem{Evstigneev2022}
Evstigneev, N., Ryabkov, O., Bocharov, A., Petrovskiy, V., Teplyakov, I.:
  Compensated summation and dot product algorithms for floating-point vectors
  on parallel architectures: Error bounds, implementation and application in
  the krylov subspace methods. Journal of Computational and Applied Mathematics
   \textbf{414},  114434 (2022).
  \doi{https://doi.org/10.1016/j.cam.2022.114434},
  \url{https://www.sciencedirect.com/science/article/pii/S0377042722002047}

\bibitem{PythonDocs}
{P}ython~{S}oftware {F}oundation: The {P}ython standard library: Built-in
  functions. \url{https://docs.python.org/3/library/functions.html}
  (2001--2025)

\bibitem{MPFR}
Fousse, L., Hanrot, G., Lef\`{e}vre, V., P\'{e}lissier, P., Zimmermann, P.:
  Mpfr: A multiple-precision binary floating-point library with correct
  rounding. ACM Trans. Math. Softw.  \textbf{33}(2),  13–es (Jun 2007).
  \doi{10.1145/1236463.1236468}, \url{https://doi.org/10.1145/1236463.1236468}

\bibitem{Frolov2000}
Frolov, A.M.: High-precision, variational, bound-state calculations in coulomb
  three-body systems. Phys. Rev. E  \textbf{62},  8740--8745 (Dec 2000).
  \doi{10.1103/PhysRevE.62.8740},
  \url{https://link.aps.org/doi/10.1103/PhysRevE.62.8740}

\bibitem{Frolov2003}
Frolov, A.M., Bailey, D.H.: Highly accurate evaluation of the few-body
  auxiliary functions and four-body integrals. Journal of Physics B: Atomic,
  Molecular and Optical Physics  \textbf{36}(9), ~1857 (apr 2003).
  \doi{10.1088/0953-4075/36/9/315},
  \url{https://dx.doi.org/10.1088/0953-4075/36/9/315}

\bibitem{GMP}
Granlund, T., Team, G.D.: {GNU} {MP} 6.0 Multiple Precision Arithmetic Library.
  Samurai Media Limited, London, GBR (2015)

\bibitem{He2001}
He, Y., Ding, C.H.Q.: Using {Accurate} {Arithmetics} to {Improve} {Numerical}
  {Reproducibility} and {Stability} in {Parallel} {Applications}. The Journal
  of Supercomputing  \textbf{18}(3),  259--277 (Mar 2001).
  \doi{10.1023/A:1008153532043}, \url{https://doi.org/10.1023/A:1008153532043}

\bibitem{QD}
Hida, Y., Li, X., Bailey, D.: Algorithms for quad-double precision floating
  point arithmetic. In: Proceedings 15th {IEEE} {Symposium} on {Computer}
  {Arithmetic}. {ARITH}-15 2001. pp. 155--162 (Jun 2001).
  \doi{10.1109/ARITH.2001.930115},
  \url{https://ieeexplore.ieee.org/document/930115}, iSSN: 1063-6889

\bibitem{Arb}
Johansson, F.: Arb: efficient arbitrary-precision midpoint-radius interval
  arithmetic. IEEE Transactions on Computers  \textbf{66},  1281--1292 (2017).
  \doi{10.1109/TC.2017.2690633}

\bibitem{Joldes2017}
Joldes, M., Muller, J.M., Popescu, V.: Tight and rigorous error bounds for
  basic building blocks of double-word arithmetic. ACM Trans. Math. Softw.
  \textbf{44}(2) (Oct 2017). \doi{10.1145/3121432},
  \url{https://doi.org/10.1145/3121432}

\bibitem{CAMPARY}
Joldes, M., Muller, J.M., Popescu, V., Tucker, W.: {CAMPARY: Cuda Multiple
  Precision Arithmetic Library and Applications}. In: {5th International
  Congress on Mathematical Software (ICMS)}. Berlin, Germany (Jul 2016),
  \url{https://hal.science/hal-01312858}

\bibitem{Kahan1965}
Kahan, W.: Pracniques: further remarks on reducing truncation errors. Commun.
  ACM  \textbf{8}(1), ~40 (Jan 1965). \doi{10.1145/363707.363723},
  \url{https://doi.org/10.1145/363707.363723}

\bibitem{KarpMarkstein1997}
Karp, A.H., Markstein, P.: High-precision division and square root. ACM Trans.
  Math. Softw.  \textbf{23}(4),  561–589 (Dec 1997).
  \doi{10.1145/279232.279237}, \url{https://doi.org/10.1145/279232.279237}

\bibitem{Knuth1969}
Knuth, D.E.: The Art of Computer Programming, Volume {II:} Seminumerical
  Algorithms. Addison-Wesley (1969),
  \url{https://www.worldcat.org/oclc/310551264}

\bibitem{Knuth1973}
Knuth, D.E.: The Art of Computer Programming, Volume {III:} Sorting and
  Searching. Addison-Wesley (1973)

\bibitem{Lauter2005}
Lauter, C.Q.: {Basic building blocks for a triple-double intermediate format}.
  Research Report RR-5702, LIP RR-2005-38, {INRIA, LIP} (Sep 2005),
  \url{https://inria.hal.science/inria-00070314}

\bibitem{XBLAS}
Li, X.S., Demmel, J.W., Bailey, D.H., Henry, G., Hida, Y., Iskandar, J., Kahan,
  W., Kang, S.Y., Kapur, A., Martin, M.C., Thompson, B.J., Tung, T., Yoo, D.J.:
  Design, implementation and testing of extended and mixed precision blas. ACM
  Trans. Math. Softw.  \textbf{28}(2),  152–205 (Jun 2002).
  \doi{10.1145/567806.567808}, \url{https://doi.org/10.1145/567806.567808}

\bibitem{GQD}
Lu, M., He, B., Luo, Q.: Supporting extended precision on graphics processors.
  In: Proceedings of the Sixth International Workshop on Data Management on New
  Hardware. pp. 19--26. DaMoN '10, ACM, New York, NY, USA (2010).
  \doi{10.1145/1869389.1869392},
  \url{http://doi.acm.org/10.1145/1869389.1869392}

\bibitem{Monniaux2008}
Monniaux, D.: The pitfalls of verifying floating-point computations. ACM Trans.
  Program. Lang. Syst.  \textbf{30}(3) (May 2008).
  \doi{10.1145/1353445.1353446}, \url{https://doi.org/10.1145/1353445.1353446}

\bibitem{MullerUlp}
Muller, J.M.: {On the definition of ulp(x)}. Research Report RR-5504, LIP
  RR-2005-09, {INRIA, LIP} (Feb 2005),
  \url{https://inria.hal.science/inria-00070503}

\bibitem{FPHandbook}
Muller, J.M., Brunie, N., De~Dinechin, F., Jeannerod, C.P., Joldes, M.,
  Lefèvre, V., Melquiond, G., Revol, N., Torres, S.: Handbook of
  {Floating}-{Point} {Arithmetic}. Springer International Publishing, Cham
  (2018). \doi{10.1007/978-3-319-76526-6},
  \url{http://link.springer.com/10.1007/978-3-319-76526-6}

\bibitem{Muller2022}
Muller, J.M., Rideau, L.: Formalization of double-word arithmetic, and comments
  on “tight and rigorous error bounds for basic building blocks of
  double-word arithmetic”. ACM Trans. Math. Softw.  \textbf{48}(1) (Feb
  2022). \doi{10.1145/3484514}, \url{https://doi.org/10.1145/3484514}

\bibitem{Moller1965}
Møller, O.: Quasi double-precision in floating point addition. BIT Numerical
  Mathematics  \textbf{5}(1),  37--50 (Mar 1965). \doi{10.1007/BF01975722},
  \url{https://doi.org/10.1007/BF01975722}

\bibitem{Neumaier1974}
Neumaier, A.: Rundungsfehleranalyse einiger verfahren zur summation endlicher
  summen. ZAMM - Journal of Applied Mathematics and Mechanics / Zeitschrift
  für Angewandte Mathematik und Mechanik  \textbf{54}(1),  39--51 (1974).
  \doi{https://doi.org/10.1002/zamm.19740540106},
  \url{https://onlinelibrary.wiley.com/doi/abs/10.1002/zamm.19740540106}

\bibitem{BitwuzlaTool}
Niemetz, A., Preiner, M.: Bitwuzla. In: Enea, C., Lal, A. (eds.) Computer Aided
  Verification - 35th International Conference, {CAV} 2023, Paris, France, July
  17-22, 2023, Proceedings, Part {II}. Lecture Notes in Computer Science, vol.
  13965, pp. 3--17. Springer (2023). \doi{10.1007/978-3-031-37703-7\_1},
  \url{https://doi.org/10.1007/978-3-031-37703-7\_1}

\bibitem{BitwuzlaAbstraction}
Niemetz, A., Preiner, M., Zohar, Y.: Scalable bit-blasting with abstractions.
  In: Gurfinkel, A., Ganesh, V. (eds.) Computer Aided Verification. pp.
  178--200. Springer Nature Switzerland, Cham (2024)

\bibitem{Rump2005}
Ogita, T., Rump, S.M., Oishi, S.: Accurate sum and dot product. SIAM Journal on
  Scientific Computing  \textbf{26}(6),  1955--1988 (2005).
  \doi{10.1137/030601818}, \url{https://doi.org/10.1137/030601818}

\bibitem{Priest1991}
Priest, D.: Algorithms for arbitrary precision floating point arithmetic. In:
  [1991] Proceedings 10th IEEE Symposium on Computer Arithmetic. pp. 132--143
  (1991). \doi{10.1109/ARITH.1991.145549}

\bibitem{JuliaDocs}
{T}he~{J}ulia {P}roject: {J}ulia standard library: Arrays.
  \url{https://docs.julialang.org/en/v0.6/stdlib/arrays} (2017)

\bibitem{Rivera2024}
Rivera, J., Franchetti, F., P\"{u}schel, M.: Floating-point tvpi abstract
  domain. Proc. ACM Program. Lang.  \textbf{8}(PLDI) (Jun 2024).
  \doi{10.1145/3656395}, \url{https://doi.org/10.1145/3656395}

\bibitem{Rump2016}
Rump, S.M., Lange, M.: On the definition of unit roundoff. BIT Numerical
  Mathematics  \textbf{56}(1),  309--317 (Mar 2016).
  \doi{10.1007/s10543-015-0554-0},
  \url{https://doi.org/10.1007/s10543-015-0554-0}

\bibitem{DoubleFloats}
Sarnoff, J.: {DoubleFloats.jl} (2024),
  \url{https://github.com/JuliaMath/DoubleFloats.jl}

\bibitem{Shewchuk1997}
Shewchuk, J.R.: Adaptive {Precision} {Floating}-{Point} {Arithmetic} and {Fast}
  {Robust} {Geometric} {Predicates}. Discrete \& Computational Geometry
  \textbf{18}(3),  305--363 (Oct 1997). \doi{10.1007/PL00009321},
  \url{https://doi.org/10.1007/PL00009321}

\bibitem{COREMATH}
Sibidanov, A., Zimmermann, P., Glondu, S.: {The CORE-MATH Project}. In: {2022
  IEEE 29th Symposium on Computer Arithmetic (ARITH)}. pp. 26--34. {IEEE},
  virtual, France (Sep 2022). \doi{10.1109/ARITH54963.2022.00014},
  \url{https://inria.hal.science/hal-03721525}

\bibitem{Veltkamp1968}
Veltkamp, G.W.: {ALGOL} procedures voor het berekenen van een inwendig product
  in dubbele precisie. Tech. Rep.~22, Technische Hogeschool Eindhoven (1968)

\bibitem{Veltkamp1969}
Veltkamp, G.W.: {ALGOL} procedures voor het rekenen in dubbele lengte. Tech.
  Rep.~21, Technische Hogeschool Eindhoven (1969)

\bibitem{MultiFloats}
Zhang, D.K.: {MultiFloats.jl} (2024),
  \url{https://github.com/dzhang314/MultiFloats.jl}

\end{thebibliography}

\clearpage
\section*{Appendix} \label{sec:Appendix}

Let $x$ and $y$ be floating-point numbers, and let $(s, e) \coloneqq \TwoSum(x, y)$. Assume that $x$, $y$, $s$, and $e$ are all normalized or zero. The following lemmas constrain the possible values of $x$, $y$, $s$, and $e$ using the SE, SETZ, and SELTZO abstractions. Note that $\TwoSum$ is a commutative operation (i.e., $\TwoSum(x, y) = \TwoSum(y, x)$), so all of the following results also hold with the roles of $x$ and $y$ are reversed. We omit symmetric statements for brevity.

\subsection*{Lemma Family Z}

Lemmas in Family Z apply when one or both of the inputs $(x, y)$ are zero. These Lemmas support our SE, SETZ, and SELTZO results. \\

\noindent \textbf{Lemma Z1:}
\begin{align*}
    \TwoSum(\mathtt{+0.0}, \mathtt{+0.0}) &= (\mathtt{+0.0}, \mathtt{+0.0}) \\
    \TwoSum(\mathtt{+0.0}, \mathtt{-0.0}) &= (\mathtt{+0.0}, \mathtt{+0.0}) \\
    \TwoSum(\mathtt{-0.0}, \mathtt{-0.0}) &= (\mathtt{-0.0}, \mathtt{+0.0})
\end{align*}

\noindent \textbf{Lemma Z2:} If $x$ is nonzero, then $\TwoSum(x, \pm\mathtt{0.0}) = (x, \mathtt{+0.0})$.

\subsection*{Lemma Family SE-I}

Lemma SE-I gives a sufficient condition for the inputs $(x, y)$ to be returned unchanged by $\TwoSum$. This Lemma supports only our SE results and is superseded by a more precise Lemma in the SETZ and SELTZO abstractions. Note that, unlike the SETZ and SELTZO abstractions, the SE abstraction is not strong enough to state necessary and sufficient conditions for $(x, y) = \TwoSum(x, y)$. \\

\noindent \textbf{Lemma SE-I:} If $\lvert e_x - e_y \rvert < p + 1$ or $\lvert e_x - e_y \rvert = p + 1$ and $s_x = s_y$, then $(x, y) = \TwoSum(x, y)$.

\subsection*{Lemma Family SE-S}

Lemmas in Family SE-S apply to SE abstract inputs with the same sign. These Lemmas support only our SE results and are superseded by more precise Lemmas in the SETZ and SELTZO abstractions. In this Lemma Family, we implicitly assume that $x$ and $y$ are nonzero with $s_x = s_y$ throughout. \\

\noindent \textbf{Lemma SE-S1:} If $e_x = e_y + p$, then one of the following statements holds:
\begin{enumerate}
    \item $s_s = s_x$, $e_x \le e_s \le e_x + 1$, $s_e \ne s_y$, and $e_y - (p - 1) \le e_e \le e_x - p$.
    \item $s = x$ and $e = y$.
\end{enumerate}

\noindent \textbf{Lemma SE-S2:} If $e_x = e_y + (p - 1)$, then one of the following statements holds:
\begin{enumerate}
    \item $s_s = s_x$, $e_x \le e_s \le e_x + 1$, and $e = \mathtt{+0.0}$.
    \item $s_s = s_x$, $e_x \le e_s \le e_x + 1$, and $e_y - (p - 1) \le e_e \le e_x - p$.
\end{enumerate}

\noindent \textbf{Lemma SE-S3:} If $e_x = e_y + (p - 2)$, then one of the following statements holds:
\begin{enumerate}
    \item $s_s = s_x$, $e_x \le e_s \le e_x + 1$, and $e = \mathtt{+0.0}$.
    \item $s_s = s_x$, $e_x \le e_s \le e_x + 1$, $s_e \ne s_y$, and $e_y - (p - 1) \le e_e \le e_x - p$.
    \item $s_s = s_x$, $e_s = e_x$, $s_e = s_y$, and $e_y - (p - 1) \le e_e \le e_x - p$.
    \item $s_s = s_x$, $e_s = e_x + 1$, $s_e = s_y$, and $e_y - (p - 1) \le e_e \le e_x - (p - 1)$.
\end{enumerate}

\noindent \textbf{Lemma SE-S4:} If $e_x > e_y$ and $e_x < e_y + (p - 2)$, then one of the following statements holds:
\begin{enumerate}
    \item $s_s = s_x$, $e_x \le e_s \le e_x + 1$, and $e = \mathtt{+0.0}$.
    \item $s_s = s_x$, $e_s = e_x$, and $e_y - (p - 1) \le e_e \le e_x - p$.
    \item $s_s = s_x$, $e_s = e_x + 1$, and $e_y - (p - 1) \le e_e \le e_x - (p - 1)$.
\end{enumerate}

\noindent \textbf{Lemma SE-S5:} If $e_x = e_y$, then one of the following statements holds:
\begin{enumerate}
    \item $s_s = s_x$, $e_s = e_x + 1$, and $e = \mathtt{+0.0}$.
    \item $s_s = s_x$, $e_s = e_x + 1$, and $e_e = e_x - (p - 1)$.
\end{enumerate}

\subsection*{Lemma Family SE-D}

Lemmas in Family SE-D apply to SE abstract inputs with different signs. These Lemmas support only our SE results and are superseded by more precise Lemmas in the SETZ and SELTZO abstractions. In this Lemma Family, we implicitly assume that $x$ and $y$ are nonzero with $s_x \ne s_y$ throughout. \\

\noindent \textbf{Lemma SE-D1:} If $e_x = e_y + (p + 1)$, then one of the following statements holds:
\begin{enumerate}
    \item $s_s = s_x$, $e_s = e_x - 1$, $s_e \ne s_y$, and $e_y - (p - 1) \le e_e \le e_x - (p + 2)$.
    \item $s = x$ and $e = y$.
\end{enumerate}

\noindent \textbf{Lemma SE-D2:} If $e_x = e_y + p$, then one of the following statements holds:
\begin{enumerate}
    \item $s_s = s_x$, $e_s = e_x - 1$, and $e = \mathtt{+0.0}$.
    \item $s_s = s_x$, $e_s = e_x - 1$, $s_e = s_y$, and $e_y - (p - 1) \le e_e \le e_x - (p + 2)$.
    \item $s_s = s_x$, $e_s = e_x - 1$, $s_e \ne s_y$, and $e_y - (p - 1) \le e_e \le e_x - (p + 1)$.
    \item $s_s = s_x$, $e_s = e_x$, $s_e \ne s_y$, and $e_y - (p - 1) \le e_e \le e_x - p$.
    \item $s = x$ and $e = y$.
\end{enumerate}

\noindent \textbf{Lemma SE-D3:} If $e_x > e_y + 1$ and $e_x < e_y + p$, then one of the following statements holds:
\begin{enumerate}
    \item $s_s = s_x$, $e_x - 1 \le e_s \le e_x$, and $e = \mathtt{+0.0}$.
    \item $s_s = s_x$, $e_s = e_x - 1$, and $e_y - (p - 1) \le e_e \le e_x - (p + 1)$.
    \item $s_s = s_x$, $e_s = e_x$, and $e_y - (p - 1) \le e_e \le e_x - p$.
\end{enumerate}

\noindent \textbf{Lemma SE-D4:} If $e_x = e_y + 1$, then one of the following statements holds:
\begin{enumerate}
    \item $s_s = s_x$, $e_x - p \le e_s \le e_x$, and $e = \mathtt{+0.0}$.
    \item $s_s = s_x$, $e_s = e_x$, and $e_e = e_x - p$.
\end{enumerate}

\noindent \textbf{Lemma SE-D5:} If $e_x = e_y$, then one of the following statements holds:
\begin{enumerate}
    \item $s = \mathtt{+0.0}$ and $e = \mathtt{+0.0}$.
    \item $e_x - (p - 1) \le e_s \le e_x - 1$ and $e = \mathtt{+0.0}$.
\end{enumerate}

\subsection*{Lemma Family SETZ-I}

Lemma SETZ-I (for ``identical'') gives necessary and sufficient conditions for the inputs $(x, y)$ to be returned unchanged by $\TwoSum$. This Lemma supports both our SETZ and SELTZO results. \\

\noindent \textbf{Lemma SETZ-I:} Assume $x$ and $y$ are nonzero. $(s, e) = (x, y)$ if and only if any of the following conditions hold:
\begin{enumerate}
    \item $e_x > e_y + (p + 1)$.
    \item $e_x = e_y + (p + 1)$ and any of the following conditions hold: $e_y = f_y$, $s_x = s_y$, or $e_x > f_x$.
    \item $e_x = e_y + p$, $e_y = f_y$, $e_x < f_x + (p - 1)$, and $s_x = s_y$ or $e_x > f_x$.
\end{enumerate}

\subsection*{Lemma Family SETZ-F}

The remaining SETZ lemmas admit simpler statements if we make a change of variables from $(s_x, e_x, \ntz_x)$ to $(s_x, e_x, f_x)$ where $f_x \coloneqq e_x - (p - \ntz_x - 1)$. We call this value the \textit{trailing exponent} of $x$. The trailing exponent is the place value of the last nonzero mantissa bit of $x$; it acts like the dual of $e_x$, bounding the mantissa from below. \\

Lemmas in Family SETZ-F apply to addends with the same trailing exponent (i.e., $f_x = f_y$). These Lemmas support our SETZ results. We assume throughout this Lemma Family that $x$ and $y$ are nonzero. \\

\noindent \textbf{Lemma SETZ-FS0:} If $s_x = s_y$, $f_x = f_y$, and $e_x > e_y + 1$, then one of the following statements holds:
\begin{enumerate}
   \item $s_s = s_x$, $e_s = e_x$, $f_x + 1 \le f_s \le e_x - 1$, and $e = \mathtt{+0.0}$.
   \item $s_s = s_x$, $e_s = e_x + 1$, $f_x + 1 \le f_s \le e_y$, and $e = \mathtt{+0.0}$.
   \item $s_s = s_x$, $e_s = e_x + 1$, $f_s = e_x + 1$, and $e = \mathtt{+0.0}$.
\end{enumerate}

\noindent \textbf{Lemma SETZ-FS1:} If $s_x = s_y$, $f_x = f_y$, and $e_x = e_y + 1$, then one of the following statements holds:
\begin{enumerate}
   \item $s_s = s_x$, $e_s = e_x$, $f_x + 1 \le f_s \le e_x - 2$, and $e = \mathtt{+0.0}$.
   \item $s_s = s_x$, $e_s = e_x + 1$, $f_x + 1 \le f_s \le e_y$, and $e = \mathtt{+0.0}$.
   \item $s_s = s_x$, $e_s = e_x + 1$, $f_s = e_x + 1$, and $e = \mathtt{+0.0}$.
\end{enumerate}

\noindent \textbf{Lemma SETZ-FS2:} If $s_x = s_y$, $f_x = f_y$, $e_x = e_y$, and $e_x > f_x$, then one of the following statements holds:
\begin{enumerate}
   \item $s_s = s_x$, $e_s = e_x + 1$, $f_x + 1 \le f_s \le e_x$, and $e = \mathtt{+0.0}$.
\end{enumerate}

\noindent \textbf{Lemma SETZ-FS3:} If $s_x = s_y$, $f_x = f_y$, $e_x = e_y$, and $e_x = f_x$, then one of the following statements holds:
\begin{enumerate}
   \item $s_s = s_x$, $e_s = e_x + 1$, $f_s = e_x + 1$, and $e = \mathtt{+0.0}$.
\end{enumerate}

\noindent \textbf{Lemma SETZ-FD0:} If $s_x \ne s_y$, $f_x = f_y$, and $e_x > e_y + 1$, then one of the following statements holds:
\begin{enumerate}
   \item $s_s = s_x$, $e_s = e_x - 1$, $f_x + 1 \le f_s \le e_y$, and $e = \mathtt{+0.0}$.
   \item $s_s = s_x$, $e_s = e_x$, $f_x + 1 \le f_s \le e_x$, and $e = \mathtt{+0.0}$.
\end{enumerate}

\noindent \textbf{Lemma SETZ-FD1:} If $s_x \ne s_y$, $f_x = f_y$, and $e_x = e_y + 1$, then one of the following statements holds:
\begin{enumerate}
   \item For each $k$ where $f_x + 1 \le k \le e_x - 1$: $s_s = s_x$, $e_s = k$, $f_x + 1 \le f_s \le k$, and $e = \mathtt{+0.0}$.
   \item $s_s = s_x$, $e_s = e_x$, $f_x + 1 \le f_s \le e_x - 2$, and $e = \mathtt{+0.0}$.
   \item $s_s = s_x$, $e_s = e_x$, $f_s = e_x$, and $e = \mathtt{+0.0}$.
\end{enumerate}

\noindent \textbf{Lemma SETZ-FD2:} If $s_x \ne s_y$, $f_x = f_y$, and $e_x = e_y$, then one of the following statements holds:
\begin{enumerate}
   \item $s = \mathtt{+0.0}$ and $e = \mathtt{+0.0}$.
   \item For each $k$ where $f_x + 1 \le k \le e_x - 1$: $f_x + 1 \le f_s \le k$, $e_s = k$, and $e = \mathtt{+0.0}$.
\end{enumerate}

\subsection*{Lemma Family SETZ-E}

Lemmas in Family SETZ-E (for ``exact'') apply to addends with different trailing exponents whose floating-point sum is exact (i.e., $e = \pm\mathtt{0.0}$). These Lemmas support our SETZ results. We assume throughout this Lemma Family that $x$ and $y$ are nonzero. \\

\noindent \textbf{Lemma SETZ-EN0:} If $s_x = s_y$ or $e_x > f_x$, $f_x > e_y$, and $e_x < f_y + p$, then one of the following statements holds:
\begin{enumerate}
   \item $s_s = s_x$, $e_s = e_x$, $f_s = f_y$, and $e = \mathtt{+0.0}$.
\end{enumerate}

\noindent \textbf{Lemma SETZ-EN1:} If $s_x \ne s_y$, and one of the following statements holds:
\begin{enumerate}
    \item $e_x = f_x$, $f_x > e_y + 1$, $e_x < f_y + (p + 1)$
    \item $e_x = f_x + 1$, $f_x = e_y$, $e_y > f_y$
\end{enumerate}
then one of the following statements holds:
\begin{enumerate}
   \item $s_s = s_x$, $e_s = e_x - 1$, $f_s = f_y$, and $e = \mathtt{+0.0}$.
\end{enumerate}

\noindent \textbf{Lemma SETZ-ESP0:} If $s_x = s_y$, either $(e_x > e_y > f_x > f_y)$ or $(e_x > e_y + 1 > f_x > f_y)$, and $e_x < f_y + (p - 1)$, then one of the following statements holds:
\begin{enumerate}
   \item $s_s = s_x$, $e_x \le e_s \le e_x + 1$, $f_s = f_y$, and $e = \mathtt{+0.0}$.
\end{enumerate}

\noindent \textbf{Lemma SETZ-ESP1:} If $s_x = s_y$, $e_x = e_y + 1$, $e_y = f_x > f_y$, and $e_x < f_y + (p - 1)$, then one of the following statements holds:
\begin{enumerate}
   \item $s_s = s_x$, $e_s = e_x + 1$, $f_s = f_y$, and $e = \mathtt{+0.0}$.
\end{enumerate}

\noindent \textbf{Lemma SETZ-ESC:} If $s_x = s_y$, $e_x > e_y$, $f_x < f_y$, and $e_x < f_x + (p - 1)$, then one of the following statements holds:
\begin{enumerate}
   \item $s_s = s_x$, $e_x \le e_s \le e_x + 1$, $f_s = f_x$, and $e = \mathtt{+0.0}$.
\end{enumerate}

\noindent \textbf{Lemma SETZ-ESS:} If $s_x = s_y$, $e_x = e_y$, $f_x < f_y$, $e_x < f_x + (p - 1)$, and $e_y < f_y + (p - 1)$, then one of the following statements holds:
\begin{enumerate}
   \item $s_s = s_x$, $e_s = e_x + 1$, $f_s = f_x$, and $e = \mathtt{+0.0}$.
\end{enumerate}

\noindent \textbf{Lemma SETZ-EDP0:} If $s_x \ne s_y$, $e_x > e_y + 1 > f_x > f_y$, and $e_x < f_y + p$, then one of the following statements holds:
\begin{enumerate}
   \item $s_s = s_x$, $e_x - 1 \le e_s \le e_x$, $f_s = f_y$, and $e = \mathtt{+0.0}$.
\end{enumerate}

\noindent \textbf{Lemma SETZ-EDP1:} If $s_x \ne s_y$, $e_x = e_y + 1$, $e_y > f_x > f_y$, and $e_x < f_y + p$, then one of the following statements holds:
\begin{enumerate}
   \item $s_s = s_x$, $f_x \le e_s \le e_x$, $f_s = f_y$, and $e = \mathtt{+0.0}$.
\end{enumerate}

\noindent \textbf{Lemma SETZ-EDP2:} If $s_x \ne s_y$, $e_x = e_y + 1 = f_x$, and $f_x > f_y + 1$, then one of the following statements holds:
\begin{enumerate}
   \item $s_s = s_x$, $f_y \le e_s \le e_x - 2$, $f_s = f_y$, and $e = \mathtt{+0.0}$.
\end{enumerate}

\noindent \textbf{Lemma SETZ-EDP3:} If $s_x \ne s_y$, $e_x = e_y + 1 = f_x = f_y + 1$, then one of the following statements holds:
\begin{enumerate}
   \item $s_s = s_x$, $f_y \le e_s \le e_x - 1$, $f_s = f_y$, and $e = \mathtt{+0.0}$.
\end{enumerate}

\noindent \textbf{Lemma SETZ-EDC0:} If $s_x \ne s_y$, $e_x > e_y + 1$, and $f_x < f_y$, then one of the following statements holds:
\begin{enumerate}
   \item $s_s = s_x$, $e_x - 1 \le e_s \le e_x$, $f_s = f_x$, and $e = \mathtt{+0.0}$.
\end{enumerate}

\noindent \textbf{Lemma SETZ-EDC1:} If $s_x \ne s_y$, $e_x = e_y + 1$, and $f_x < f_y$, then one of the following statements holds:
\begin{enumerate}
   \item $s_s = s_x$, $f_y \le e_s \le e_x$, $f_s = f_x$, and $e = \mathtt{+0.0}$.
\end{enumerate}

\noindent \textbf{Lemma SETZ-EDC2:} If $s_x \ne s_y$, $e_x = e_y = f_y$, and $f_x < f_y$, then one of the following statements holds:
\begin{enumerate}
   \item $s_s = s_x$, $f_x \le e_s \le e_x - 1$, $f_s = f_x$, and $e = \mathtt{+0.0}$.
\end{enumerate}

\noindent \textbf{Lemma SETZ-EDS0:} If $s_x \ne s_y$, $e_x = e_y$, $f_x < f_y$, $e_x > f_x + 1$, and $e_y > f_y + 1$, then one of the following statements holds:
\begin{enumerate}
   \item $f_x \le e_s \le e_x - 1$, $f_s = f_x$, and $e = \mathtt{+0.0}$.
\end{enumerate}

\noindent \textbf{Lemma SETZ-EDS1:} If $s_x \ne s_y$, $e_x = e_y$, $e_x > f_x + 1$, and $e_y = f_y + 1$, then one of the following statements holds:
\begin{enumerate}
   \item $f_x \le e_s \le e_x - 2$, $f_s = f_x$, and $e = \mathtt{+0.0}$.
\end{enumerate}

\subsection*{Lemma Family SETZ-O}

Lemmas in Family SETZ-O (for ``overlap'') apply to addends with completely overlapping mantissas whose floating-point sum has nonzero error. These Lemmas support both our SETZ and SELTZO results. We assume throughout this Lemma Family that $x$ and $y$ are nonzero. \\

\noindent \textbf{Lemma SETZ-O0:} If $s_x = s_y$, $e_x = f_x + (p - 1)$, and $e_x > e_y > f_y > f_x$, then one of the following statements holds:
\begin{enumerate}
   \item $s_s = s_x$, $e_s = e_x$, $f_s = f_x$, and $e = \mathtt{+0.0}$.
   \item $s_s = s_x$, $e_s = e_x + 1$, $e_x - (p - 3) \le f_s \le e_y$, $f_x \le e_e \le e_x - (p - 1)$, and $f_e = f_x$.
   \item $s_s = s_x$, $e_s = e_x + 1$, $f_s = e_x + 1$, $s_e = s_y$, $f_x \le e_e \le e_x - (p - 1)$, and $f_e = f_x$.
\end{enumerate}

\noindent \textbf{Lemma SETZ-O1:} If $s_x = s_y$, $e_x = f_x + (p - 1)$, and $e_x > e_y = f_y > f_x + 1$, then one of the following statements holds:
\begin{enumerate}
   \item $s_s = s_x$, $e_s = e_x$, $f_s = f_x$, and $e = \mathtt{+0.0}$.
   \item $s_s = s_x$, $e_s = e_x + 1$, $e_x - (p - 3) \le f_s \le e_y - 1$, $f_x \le e_e \le e_x - (p - 1)$, and $f_e = f_x$.
   \item $s_s = s_x$, $e_s = e_x + 1$, $f_s = e_y$, $s_e \ne s_y$, $f_x \le e_e \le e_x - (p - 1)$, and $f_e = f_x$.
   \item $s_s = s_x$, $e_s = e_x + 1$, $f_s = e_x + 1$, $s_e = s_y$, $f_x \le e_e \le e_x - (p - 1)$, and $f_e = f_x$.
\end{enumerate}

\noindent \textbf{Lemma SETZ-O2:} If $s_x = s_y$, $e_x = f_x + (p - 1)$, and $e_y = f_y = f_x + 1$, then one of the following statements holds:
\begin{enumerate}
   \item $s_s = s_x$, $e_s = e_x$, $f_s = f_x$, and $e = \mathtt{+0.0}$.
   \item $s_s = s_x$, $e_s = e_x + 1$, $f_s = e_x + 1$, $s_e = s_y$, $f_x \le e_e \le e_x - (p - 1)$, and $f_e = f_x$.
\end{enumerate}

\subsection*{Lemma Family SETZ-1}

These Lemmas support both our SETZ and SELTZO results. We assume throughout this Lemma Family that $x$ and $y$ are nonzero. \\

\noindent \textbf{Lemma SETZ-1:} If $e_x < e_y + p$, $e_x > f_y + p$, $f_x > e_y + 1$, and $e_x > f_x$ or $s_x = s_y$, then one of the following statements holds:
\begin{enumerate}
   \item $s_s = s_x$, $e_s = e_x$, $e_x - (p - 1) \le f_s \le e_y - 1$, $f_y \le e_e \le e_x - (p + 1)$, and $f_e = f_y$.
   \item $s_s = s_x$, $e_s = e_x$, $f_s = e_y$, $s_e = s_y$, $f_y \le e_e \le e_x - (p + 1)$, and $f_e = f_y$.
   \item $s_s = s_x$, $e_s = e_x$, $f_s = e_y + 1$, $s_e \ne s_y$, $f_y \le e_e \le e_x - (p + 1)$, and $f_e = f_y$.
\end{enumerate}

\noindent \textbf{Lemma SETZ-1A:} If $e_x = e_y + p$, $e_x > f_y + p$, $f_x > e_y + 1$, and $e_x > f_x$ or $s_x = s_y$, then one of the following statements holds:
\begin{enumerate}
   \item $s_s = s_x$, $e_s = e_x$, $f_s = e_y + 1$, $s_e \ne s_y$, $f_y \le e_e \le e_x - (p + 1)$, and $f_e = f_y$.
\end{enumerate}

\noindent \textbf{Lemma SETZ-1B0:} If $e_x < e_y + (p - 1)$, $e_x = f_y + p$, $f_x > e_y + 1$, and $e_x > f_x$ or $s_x = s_y$, then one of the following statements holds:
\begin{enumerate}
   \item $s_s = s_x$, $e_s = e_x$, $e_x - (p - 2) \le f_s \le e_y - 1$, $f_y \le e_e \le e_x - p$, and $f_e = f_y$.
   \item $s_s = s_x$, $e_s = e_x$, $f_s = e_y$, $s_e = s_y$, $f_y \le e_e \le e_x - p$, and $f_e = f_y$.
   \item $s_s = s_x$, $e_s = e_x$, $f_s = e_y + 1$, $s_e \ne s_y$, $f_y \le e_e \le e_x - p$, and $f_e = f_y$.
\end{enumerate}

\noindent \textbf{Lemma SETZ-1B1:} If $e_x = e_y + (p - 1)$, $e_x = f_y + p$, $f_x > e_y + 1$, and $e_x > f_x$ or $s_x = s_y$, then one of the following statements holds:
\begin{enumerate}
   \item $s_s = s_x$, $e_s = e_x$, $f_s = e_y + 1$, $s_e \ne s_y$, $f_y \le e_e \le e_x - p$, and $f_e = f_y$.
\end{enumerate}

\subsection*{Lemma Family SETZ-2}

These Lemmas support both our SETZ and SELTZO results. We assume throughout this Lemma Family that $x$ and $y$ are nonzero. \\

\noindent \textbf{Lemma SETZ-2:} If $s_x = s_y$, $e_x > f_y + p$, and $f_x < e_y$, then one of the following statements holds:
\begin{enumerate}
   \item $s_s = s_x$, $e_s = e_x$, $e_x - (p - 1) \le f_s \le e_x - 1$, $f_y \le e_e \le e_x - (p + 1)$, and $f_e = f_y$.
   \item $s_s = s_x$, $e_s = e_x + 1$, $e_x - (p - 2) \le f_s \le e_y$, $f_y \le e_e \le e_x - p$, and $f_e = f_y$.
   \item $s_s = s_x$, $e_s = e_x + 1$, $f_s = e_x + 1$, $s_e \ne s_y$, $f_y \le e_e \le e_x - (p + 1)$, and $f_e = f_y$.
   \item $s_s = s_x$, $e_s = e_x + 1$, $f_s = e_x + 1$, $s_e = s_y$, $f_y \le e_e \le e_x - p$, and $f_e = f_y$.
\end{enumerate}

\noindent \textbf{Lemma SETZ-2A0:} If $s_x = s_y$, $e_x = f_y + p$, $f_x < e_y$, and $e_y < f_y + (p - 1)$, then one of the following statements holds:
\begin{enumerate}
   \item $s_s = s_x$, $e_s = e_x$, $e_x - (p - 2) \le f_s \le e_x - 1$, $f_y \le e_e \le e_x - p$, and $f_e = f_y$.
   \item $s_s = s_x$, $e_s = e_x + 1$, $e_x - (p - 2) \le f_s \le e_y$, $f_y \le e_e \le e_x - p$, and $f_e = f_y$.
   \item $s_s = s_x$, $e_s = e_x + 1$, $f_s = e_x + 1$, $f_y \le e_e \le e_x - p$, and $f_e = f_y$.
\end{enumerate}

\noindent \textbf{Lemma SETZ-2A1:} If $s_x = s_y$, $e_x = f_y + p$, $f_x + 1 < e_y$, and $e_y = f_y + (p - 1)$, then one of the following statements holds:
\begin{enumerate}
   \item $s_s = s_x$, $e_s = e_x$, $e_x - (p - 2) \le f_s \le e_x - 2$, $f_y \le e_e \le e_x - p$, and $f_e = f_y$.
   \item $s_s = s_x$, $e_s = e_x + 1$, $e_x - (p - 2) \le f_s \le e_y$, $f_y \le e_e \le e_x - p$, and $f_e = f_y$.
   \item $s_s = s_x$, $e_s = e_x + 1$, $f_s = e_x + 1$, $f_y \le e_e \le e_x - p$, and $f_e = f_y$.
\end{enumerate}

\noindent \textbf{Lemma SETZ-2A2:} If $s_x = s_y$, $e_x = f_y + p$, $f_x + 1 = e_y$, and $e_y = f_y + (p - 1)$, then one of the following statements holds:
\begin{enumerate}
   \item $s_s = s_x$, $e_s = e_x$, $e_x - (p - 2) \le f_s \le e_y - 2$, $f_y \le e_e \le e_x - p$, and $f_e = f_y$.
   \item $s_s = s_x$, $e_s = e_x$, $f_s = e_y - 1$, $s_e = s_y$, $f_y \le e_e \le e_x - p$, and $f_e = f_y$.
   \item $s_s = s_x$, $e_s = e_x + 1$, $e_x - (p - 2) \le f_s \le e_y$, $f_y \le e_e \le e_x - p$, and $f_e = f_y$.
   \item $s_s = s_x$, $e_s = e_x + 1$, $f_s = e_x + 1$, $f_y \le e_e \le e_x - p$, and $f_e = f_y$.
\end{enumerate}

\noindent \textbf{Lemma SETZ-2B0:} If $s_x = s_y$, $e_x > f_y + p$, $f_x = e_y$, and $e_x < f_x + (p - 1)$, then one of the following statements holds:
\begin{enumerate}
   \item $s_s = s_x$, $e_s = e_x$, $e_x - (p - 1) \le f_s \le e_y - 1$, $f_y \le e_e \le e_x - (p + 1)$, and $f_e = f_y$.
   \item $s_s = s_x$, $e_s = e_x$, $f_s = e_y$, $s_e \ne s_y$, $f_y \le e_e \le e_x - (p + 1)$, and $f_e = f_y$.
   \item $s_s = s_x$, $e_s = e_x$, $e_y + 1 \le f_s \le e_x - 1$, $s_e = s_y$, $f_y \le e_e \le e_x - (p + 1)$, and $f_e = f_y$.
   \item $s_s = s_x$, $e_s = e_x + 1$, $e_x - (p - 2) \le f_s \le e_y - 1$, $f_y \le e_e \le e_x - p$, and $f_e = f_y$.
   \item $s_s = s_x$, $e_s = e_x + 1$, $f_s = e_y$, $s_e \ne s_y$, $f_y \le e_e \le e_x - p$, and $f_e = f_y$.
   \item $s_s = s_x$, $e_s = e_x + 1$, $f_s = e_x + 1$, $s_e = s_y$, $f_y \le e_e \le e_x - p$, and $f_e = f_y$.
\end{enumerate}

\noindent \textbf{Lemma SETZ-2B1:} If $s_x = s_y$, $e_x > f_y + p$, $f_x = e_y$, and $e_x = f_x + (p - 1)$, then one of the following statements holds:
\begin{enumerate}
   \item $s_s = s_x$, $e_s = e_x$, $f_s = e_y$, $s_e \ne s_y$, $f_y \le e_e \le e_x - (p + 1)$, and $f_e = f_y$.
   \item $s_s = s_x$, $e_s = e_x$, $e_y + 1 \le f_s \le e_x - 1$, $s_e = s_y$, $f_y \le e_e \le e_x - (p + 1)$, and $f_e = f_y$.
   \item $s_s = s_x$, $e_s = e_x + 1$, $f_s = e_x + 1$, $s_e = s_y$, $f_y \le e_e \le e_x - p$, and $f_e = f_y$.
\end{enumerate}

\noindent \textbf{Lemma SETZ-2C0:} If $s_x = s_y$, $e_x = f_y + (p - 1)$, $f_x < e_y$, $e_x < f_x + (p - 1)$, and $e_y < f_y + (p - 1)$, then one of the following statements holds:
\begin{enumerate}
   \item $s_s = s_x$, $e_s = e_x$, $f_s = f_y$, and $e = \mathtt{+0.0}$.
   \item $s_s = s_x$, $e_s = e_x + 1$, $e_x - (p - 3) \le f_s \le e_y$, $f_y \le e_e \le e_x - (p - 1)$, and $f_e = f_y$.
   \item $s_s = s_x$, $e_s = e_x + 1$, $f_s = e_x + 1$, $s_e = s_y$, $f_y \le e_e \le e_x - (p - 1)$, and $f_e = f_y$.
\end{enumerate}

\noindent \textbf{Lemma SETZ-2C1:} If $s_x = s_y$, $e_x = f_y + (p - 1)$, $f_x < e_y$, $e_x < f_x + (p - 1)$, and $e_y = f_y + (p - 1)$, then one of the following statements holds:
\begin{enumerate}
   \item $s_s = s_x$, $e_s = e_x + 1$, $e_x - (p - 3) \le f_s \le e_y$, $f_y \le e_e \le e_x - (p - 1)$, and $f_e = f_y$.
\end{enumerate}

\noindent \textbf{Lemma SETZ-2D0:} If $s_x = s_y$, $e_x > f_y + p$, $f_x = e_y + 1$, and $e_x < f_x + (p - 1)$, then one of the following statements holds:
\begin{enumerate}
   \item $s_s = s_x$, $e_s = e_x$, $e_x - (p - 1) \le f_s \le e_y - 1$, $f_y \le e_e \le e_x - (p + 1)$, and $f_e = f_y$.
   \item $s_s = s_x$, $e_s = e_x$, $f_s = e_y$, $s_e = s_y$, $f_y \le e_e \le e_x - (p + 1)$, and $f_e = f_y$.
   \item $s_s = s_x$, $e_s = e_x$, $e_y + 2 \le f_s \le e_x - 1$, $s_e \ne s_y$, $f_y \le e_e \le e_x - (p + 1)$, and $f_e = f_y$.
   \item $s_s = s_x$, $e_s = e_x + 1$, $f_s = e_x + 1$, $s_e \ne s_y$, $f_y \le e_e \le e_x - (p + 1)$, and $f_e = f_y$.
\end{enumerate}

\noindent \textbf{Lemma SETZ-2D1:} If $s_x = s_y$, $e_x > f_y + p$, $f_x = e_y + 1$, and $e_x = f_x + (p - 1)$, then one of the following statements holds:
\begin{enumerate}
   \item $s_s = s_x$, $e_s = e_x$, $e_y + 2 \le f_s \le e_x - 1$, $s_e \ne s_y$, $f_y \le e_e \le e_x - (p + 1)$, and $f_e = f_y$.
   \item $s_s = s_x$, $e_s = e_x + 1$, $f_s = e_x + 1$, $s_e \ne s_y$, $f_y \le e_e \le e_x - (p + 1)$, and $f_e = f_y$.
\end{enumerate}

\noindent \textbf{Lemma SETZ-2AB0:} If $s_x = s_y$, $e_x = f_y + p$, $f_x = e_y$, $e_x < f_x + (p - 1)$, and $e_y < f_y + (p - 1)$, then one of the following statements holds:
\begin{enumerate}
   \item $s_s = s_x$, $e_s = e_x$, $e_x - (p - 2) \le f_s \le e_y - 1$, $f_y \le e_e \le e_x - p$, and $f_e = f_y$.
   \item $s_s = s_x$, $e_s = e_x$, $f_s = e_y$, $s_e \ne s_y$, $f_y \le e_e \le e_x - p$, and $f_e = f_y$.
   \item $s_s = s_x$, $e_s = e_x$, $e_y + 1 \le f_s \le e_x - 1$, $s_e = s_y$, $f_y \le e_e \le e_x - p$, and $f_e = f_y$.
   \item $s_s = s_x$, $e_s = e_x + 1$, $e_x - (p - 2) \le f_s \le e_y - 1$, $f_y \le e_e \le e_x - p$, and $f_e = f_y$.
   \item $s_s = s_x$, $e_s = e_x + 1$, $f_s = e_y$, $s_e \ne s_y$, $f_y \le e_e \le e_x - p$, and $f_e = f_y$.
   \item $s_s = s_x$, $e_s = e_x + 1$, $f_s = e_x + 1$, $s_e = s_y$, $f_y \le e_e \le e_x - p$, and $f_e = f_y$.
\end{enumerate}

\noindent \textbf{Lemma SETZ-2AB1:} If $s_x = s_y$, $e_x = f_y + p$, $f_x = e_y$, and $e_x = f_x + (p - 1)$, then one of the following statements holds:
\begin{enumerate}
   \item $s_s = s_x$, $e_s = e_x$, $e_y + 1 \le f_s \le e_x - 1$, $s_e = s_y$, $f_y \le e_e \le e_x - p$, and $f_e = f_y$.
   \item $s_s = s_x$, $e_s = e_x + 1$, $f_s = e_x + 1$, $s_e = s_y$, $f_y \le e_e \le e_x - p$, and $f_e = f_y$.
\end{enumerate}

\noindent \textbf{Lemma SETZ-2AB2:} If $s_x = s_y$, $e_x = f_y + p$, $f_x = e_y$, and $e_y = f_y + (p - 1)$, then one of the following statements holds:
\begin{enumerate}
   \item $s_s = s_x$, $e_s = e_x + 1$, $e_x - (p - 2) \le f_s \le e_y - 1$, $f_y \le e_e \le e_x - p$, and $f_e = f_y$.
   \item $s_s = s_x$, $e_s = e_x + 1$, $f_s = e_y$, $s_e \ne s_y$, $f_y \le e_e \le e_x - p$, and $f_e = f_y$.
   \item $s_s = s_x$, $e_s = e_x + 1$, $f_s = e_x + 1$, $s_e = s_y$, $f_y \le e_e \le e_x - p$, and $f_e = f_y$.
\end{enumerate}

\noindent \textbf{Lemma SETZ-2BC0:} If $s_x = s_y$, $e_x = f_y + (p - 1)$, $f_x = e_y$, $e_y > f_y + 1$, and $e_y < f_y + (p - 2)$, then one of the following statements holds:
\begin{enumerate}
   \item $s_s = s_x$, $e_s = e_x$, $f_s = f_y$, and $e = \mathtt{+0.0}$.
   \item $s_s = s_x$, $e_s = e_x + 1$, $e_x - (p - 3) \le f_s \le e_y - 1$, $f_y \le e_e \le e_x - (p - 1)$, and $f_e = f_y$.
   \item $s_s = s_x$, $e_s = e_x + 1$, $f_s = e_y$, $s_e \ne s_y$, $f_y \le e_e \le e_x - (p - 1)$, and $f_e = f_y$.
   \item $s_s = s_x$, $e_s = e_x + 1$, $f_s = e_x + 1$, $s_e = s_y$, $f_y \le e_e \le e_x - (p - 1)$, and $f_e = f_y$.
\end{enumerate}

\noindent \textbf{Lemma SETZ-2BC1:} If $s_x = s_y$, $e_x = f_y + (p - 1)$, $f_x = e_y$, and $e_y > f_y + (p - 3)$, then one of the following statements holds:
\begin{enumerate}
   \item $s_s = s_x$, $e_s = e_x + 1$, $e_x - (p - 3) \le f_s \le e_y - 1$, $f_y \le e_e \le e_x - (p - 1)$, and $f_e = f_y$.
   \item $s_s = s_x$, $e_s = e_x + 1$, $f_s = e_y$, $s_e \ne s_y$, $f_y \le e_e \le e_x - (p - 1)$, and $f_e = f_y$.
   \item $s_s = s_x$, $e_s = e_x + 1$, $f_s = e_x + 1$, $s_e = s_y$, $f_y \le e_e \le e_x - (p - 1)$, and $f_e = f_y$.
\end{enumerate}

\noindent \textbf{Lemma SETZ-2BC2:} If $s_x = s_y$, $e_x = f_y + (p - 1)$, $f_x = e_y$, and $e_y = f_y + 1$, then one of the following statements holds:
\begin{enumerate}
   \item $s_s = s_x$, $e_s = e_x$, $f_s = f_y$, and $e = \mathtt{+0.0}$.
   \item $s_s = s_x$, $e_s = e_x + 1$, $f_s = e_x + 1$, $s_e = s_y$, $f_y \le e_e \le e_x - (p - 1)$, and $f_e = f_y$.
\end{enumerate}

\noindent \textbf{Lemma SETZ-2AD0:} If $s_x = s_y$, $e_x = f_y + p$, $f_x = e_y + 1$, and $e_x < f_x + (p - 2)$, then one of the following statements holds:
\begin{enumerate}
   \item $s_s = s_x$, $e_s = e_x$, $e_x - (p - 2) \le f_s \le e_y - 1$, $f_y \le e_e \le e_x - p$, and $f_e = f_y$.
   \item $s_s = s_x$, $e_s = e_x$, $f_s = e_y$, $s_e = s_y$, $f_y \le e_e \le e_x - p$, and $f_e = f_y$.
   \item $s_s = s_x$, $e_s = e_x$, $e_y + 2 \le f_s \le e_x - 1$, $s_e \ne s_y$, $f_y \le e_e \le e_x - p$, and $f_e = f_y$.
   \item $s_s = s_x$, $e_s = e_x + 1$, $f_s = e_x + 1$, $s_e \ne s_y$, $f_y \le e_e \le e_x - p$, and $f_e = f_y$.
\end{enumerate}

\noindent \textbf{Lemma SETZ-2AD1:} If $s_x = s_y$, $e_x = f_y + p$, $f_x = e_y + 1$, and $e_x > f_x + (p - 3)$, then one of the following statements holds:
\begin{enumerate}
   \item $s_s = s_x$, $e_s = e_x$, $e_y + 2 \le f_s \le e_x - 1$, $s_e \ne s_y$, $f_y \le e_e \le e_x - p$, and $f_e = f_y$.
   \item $s_s = s_x$, $e_s = e_x + 1$, $f_s = e_x + 1$, $s_e \ne s_y$, $f_y \le e_e \le e_x - p$, and $f_e = f_y$.
\end{enumerate}

\subsection*{Lemma Family SETZ-3}

These Lemmas support both our SETZ and SELTZO results. We assume throughout this Lemma Family that $x$ and $y$ are nonzero. \\

\noindent \textbf{Lemma SETZ-3:} If $s_x \ne s_y$, $e_x > f_y + (p + 1)$, and $f_x < e_y$, then one of the following statements holds:
\begin{enumerate}
    \item $s_s = s_x$, $e_s = e_x - 1$, $e_x - p \le f_s \le e_y$, $f_y \le e_e \le e_x - (p + 2)$, and $f_e = f_y$.
    \item $s_s = s_x$, $e_s = e_x$, $e_x - (p - 1) \le f_s \le e_x - 1$, $f_y \le e_e \le e_x - (p + 1)$, and $f_e = f_y$.
    \item $s_s = s_x$, $e_s = e_x$, $f_s = e_x$, $s_e = s_y$, $f_y \le e_e \le e_x - (p + 2)$, and $f_e = f_y$.
    \item $s_s = s_x$, $e_s = e_x$, $f_s = e_x$, $s_e \ne s_y$, $f_y \le e_e \le e_x - (p + 1)$, and $f_e = f_y$.
\end{enumerate}

\noindent \textbf{Lemma SETZ-3A:} If $s_x \ne s_y$, $e_x = f_y + (p + 1)$, and $f_x < e_y$, then one of the following statements holds:
\begin{enumerate}
    \item $s_s = s_x$, $e_s = e_x - 1$, $e_x - (p - 1) \le f_s \le e_y$, $f_y \le e_e \le e_x - (p + 1)$, and $f_e = f_y$.
    \item $s_s = s_x$, $e_s = e_x$, $e_x - (p - 1) \le f_s \le e_x$, $f_y \le e_e \le e_x - (p + 1)$, and $f_e = f_y$.
\end{enumerate}

\noindent \textbf{Lemma SETZ-3B:} If $s_x \ne s_y$, $e_x > f_y + (p + 1)$, and $f_x = e_y$, then one of the following statements holds:
\begin{enumerate}
    \item $s_s = s_x$, $e_s = e_x - 1$, $e_x - p \le f_s \le e_y - 1$, $f_y \le e_e \le e_x - (p + 2)$, and $f_e = f_y$.
    \item $s_s = s_x$, $e_s = e_x - 1$, $f_s = e_y$, $s_e \ne s_y$, $f_y \le e_e \le e_x - (p + 2)$, and $f_e = f_y$.
    \item $s_s = s_x$, $e_s = e_x$, $e_x - (p - 1) \le f_s \le e_y - 1$, $f_y \le e_e \le e_x - (p + 1)$, and $f_e = f_y$.
    \item $s_s = s_x$, $e_s = e_x$, $f_s = e_y$, $s_e \ne s_y$, $f_y \le e_e \le e_x - (p + 1)$, and $f_e = f_y$.
    \item $s_s = s_x$, $e_s = e_x$, $e_y + 1 \le f_s \le e_x - 1$, $s_e = s_y$, $f_y \le e_e \le e_x - (p + 1)$, and $f_e = f_y$.
    \item $s_s = s_x$, $e_s = e_x$, $f_s = e_x$, $s_e = s_y$, $f_y \le e_e \le e_x - (p + 2)$, and $f_e = f_y$.
\end{enumerate}

\noindent \textbf{Lemma SETZ-3C0:} If $s_x \ne s_y$, $e_x = f_y + p$, $f_x < e_y$, and $e_y < f_y + (p - 1)$, then one of the following statements holds:
\begin{enumerate}
    \item $s_s = s_x$, $e_s = e_x - 1$, $f_s = f_y$, and $e = \mathtt{+0.0}$.
    \item $s_s = s_x$, $e_s = e_x$, $e_x - (p - 2) \le f_s \le e_x - 1$, $f_y \le e_e \le e_x - p$, and $f_e = f_y$.
    \item $s_s = s_x$, $e_s = e_x$, $f_s = e_x$, $s_e \ne s_y$, $f_y \le e_e \le e_x - p$, and $f_e = f_y$.
\end{enumerate}

\noindent \textbf{Lemma SETZ-3C1:} If $s_x \ne s_y$, $e_x = f_y + p$, $f_x + 1 < e_y$, and $e_y = f_y + (p - 1)$, then one of the following statements holds:
\begin{enumerate}
    \item $s_s = s_x$, $f_x \le e_s \le e_x - 1$, $f_s = f_y$, and $e = \mathtt{+0.0}$.
    \item $s_s = s_x$, $e_s = e_x$, $e_x - (p - 2) \le f_s \le e_x - 2$, $f_y \le e_e \le e_x - p$, and $f_e = f_y$.
    \item $s_s = s_x$, $e_s = e_x$, $f_s = e_x$, $s_e \ne s_y$, $f_y \le e_e \le e_x - p$, and $f_e = f_y$.
\end{enumerate}

\noindent \textbf{Lemma SETZ-3C2:} If $s_x \ne s_y$, $e_x = f_y + p$, $f_x + 1 = e_y$, and $e_y = f_y + (p - 1)$, then one of the following statements holds:
\begin{enumerate}
    \item $s_s = s_x$, $e_x - 2 \le e_s \le e_x - 1$, $f_s = f_y$, and $e = \mathtt{+0.0}$.
    \item $s_s = s_x$, $e_s = e_x$, $e_x - (p - 2) \le f_s \le e_y - 2$, $f_y \le e_e \le e_x - p$, and $f_e = f_y$.
    \item $s_s = s_x$, $e_s = e_x$, $f_s = e_y - 1$, $s_e = s_y$, $f_y \le e_e \le e_x - p$, and $f_e = f_y$.
    \item $s_s = s_x$, $e_s = e_x$, $f_s = e_x$, $s_e \ne s_y$, $f_y \le e_e \le e_x - p$, and $f_e = f_y$.
\end{enumerate}

\noindent \textbf{Lemma SETZ-3D0:} If $s_x \ne s_y$, $e_x > f_y + p$, $f_x = e_y + 1$, and $e_x < f_x + (p - 1)$, then one of the following statements holds:
\begin{enumerate}
   \item $s_s = s_x$, $e_s = e_x$, $e_x - (p - 1) \le f_s \le e_y - 1$, $f_y \le e_e \le e_x - (p + 1)$, and $f_e = f_y$.
   \item $s_s = s_x$, $e_s = e_x$, $f_s = e_y$, $s_e = s_y$, $f_y \le e_e \le e_x - (p + 1)$, and $f_e = f_y$.
   \item $s_s = s_x$, $e_s = e_x$, $e_y + 2 \le f_s \le e_x$, $s_e \ne s_y$, $f_y \le e_e \le e_x - (p + 1)$, and $f_e = f_y$.
\end{enumerate}

\noindent \textbf{Lemma SETZ-3D1:} If $s_x \ne s_y$, $e_x > f_y + p$, $f_x = e_y + 1$, and $e_x = f_x + (p - 1)$, then one of the following statements holds:
\begin{enumerate}
   \item $s_s = s_x$, $e_s = e_x$, $e_y + 2 \le f_s \le e_x$, $s_e \ne s_y$, $f_y \le e_e \le e_x - (p + 1)$, and $f_e = f_y$.
\end{enumerate}

\noindent \textbf{Lemma SETZ-3AB:} If $s_x \ne s_y$, $e_x = f_y + (p + 1)$, and $f_x = e_y$, then one of the following statements holds:
\begin{enumerate}
   \item $s_s = s_x$, $e_s = e_x - 1$, $e_x - (p - 1) \le f_s \le e_y - 1$, $f_y \le e_e \le e_x - (p + 1)$, and $f_e = f_y$.
   \item $s_s = s_x$, $e_s = e_x - 1$, $f_s = e_y$, $s_e \ne s_y$, $f_y \le e_e \le e_x - (p + 1)$, and $f_e = f_y$.
   \item $s_s = s_x$, $e_s = e_x$, $e_x - (p - 1) \le f_s \le e_y - 1$, $f_y \le e_e \le e_x - (p + 1)$, and $f_e = f_y$.
   \item $s_s = s_x$, $e_s = e_x$, $f_s = e_y$, $s_e \ne s_y$, $f_y \le e_e \le e_x - (p + 1)$, and $f_e = f_y$.
   \item $s_s = s_x$, $e_s = e_x$, $e_y + 1 \le f_s \le e_x$, $s_e = s_y$, $f_y \le e_e \le e_x - (p + 1)$, and $f_e = f_y$.
\end{enumerate}

\noindent \textbf{Lemma SETZ-3BC0:} If $s_x \ne s_y$, $e_x = f_y + p$, $f_x = e_y$, $e_x > f_x + 1$, and $e_y > f_y + 1$, then one of the following statements holds:
\begin{enumerate}
   \item $s_s = s_x$, $e_s = e_x - 1$, $f_s = f_y$, and $e = \mathtt{+0.0}$.
   \item $s_s = s_x$, $e_s = e_x$, $e_x - (p - 2) \le f_s \le e_y - 1$, $f_y \le e_e \le e_x - p$, and $f_e = f_y$.
   \item $s_s = s_x$, $e_s = e_x$, $f_s = e_y$, $s_e \ne s_y$, $f_y \le e_e \le e_x - p$, and $f_e = f_y$.
   \item $s_s = s_x$, $e_s = e_x$, $e_y + 1 \le f_s \le e_x - 1$, $s_e = s_y$, $f_y \le e_e \le e_x - p$, and $f_e = f_y$.
\end{enumerate}

\noindent \textbf{Lemma SETZ-3BC1:} If $s_x \ne s_y$, $e_x = f_y + p$, $f_x = e_y$, and $e_y = f_y + 1$, then one of the following statements holds:
\begin{enumerate}
   \item $s_s = s_x$, $e_s = e_x - 1$, $f_s = f_y$, and $e = \mathtt{+0.0}$.
   \item $s_s = s_x$, $e_s = e_x$, $e_y + 1 \le f_s \le e_x - 1$, $s_e = s_y$, $f_y \le e_e \le e_x - p$, and $f_e = f_y$.
\end{enumerate}

\noindent \textbf{Lemma SETZ-3CD0:} If $s_x \ne s_y$, $e_x = f_y + p$, $f_x = e_y + 1$, $e_x > f_x$, and $e_y > f_y + 1$, then one of the following statements holds:
\begin{enumerate}
   \item $s_s = s_x$, $e_s = e_x$, $e_x - (p - 2) \le f_s \le e_y - 1$, $f_y \le e_e \le e_x - p$, and $f_e = f_y$.
   \item $s_s = s_x$, $e_s = e_x$, $f_s = e_y$, $s_e = s_y$, $f_y \le e_e \le e_x - p$, and $f_e = f_y$.
   \item $s_s = s_x$, $e_s = e_x$, $e_y + 2 \le f_s \le e_x$, $s_e \ne s_y$, $f_y \le e_e \le e_x - p$, and $f_e = f_y$.
\end{enumerate}

\noindent \textbf{Lemma SETZ-3CD1:} If $s_x \ne s_y$, $e_x = f_y + p$, $f_x = e_y + 1$, and $e_y < f_y + 2$, then one of the following statements holds:
\begin{enumerate}
   \item $s_s = s_x$, $e_s = e_x$, $e_y + 2 \le f_s \le e_x$, $s_e \ne s_y$, $f_y \le e_e \le e_x - p$, and $f_e = f_y$.
\end{enumerate}

\subsection*{Lemma Family SETZ-4}

These Lemmas support both our SETZ and SELTZO results. We assume throughout this Lemma Family that $x$ and $y$ are nonzero. \\

\noindent \textbf{Lemma SETZ-4:} If $s_x \ne s_y$, $e_x > f_y + (p + 1)$, $f_x < e_y + (p + 1)$, and $e_x = f_x$, then one of the following statements holds:
\begin{enumerate}
    \item $s_s = s_x$, $e_s = e_x - 1$, $e_x - p \le f_s \le e_y - 1$, $f_y \le e_e \le e_x - (p + 2)$, and $f_e = f_y$.
    \item $s_s = s_x$, $e_s = e_x - 1$, $f_s = e_y$, $s_e = s_y$, $f_y \le e_e \le e_x - (p + 2)$, and $f_e = f_y$.
    \item $s_s = s_x$, $e_s = e_x - 1$, $f_y + 1$, $s_e \ne s_y$, $f_y \le e_e \le e_x - (p + 2)$, and $f_e = f_y$.
\end{enumerate}

\noindent \textbf{Lemma SETZ-4A0:} If $s_x \ne s_y$, $e_x = f_y + (p + 1)$, $f_x < e_y + p$, and $e_x = f_x$, then one of the following statements holds:
\begin{enumerate}
    \item $s_s = s_x$, $e_s = e_x - 1$, $e_x - (p - 1) \le f_s \le e_y - 1$, $f_y \le e_e \le e_x - (p + 1)$, and $f_e = f_y$.
    \item $s_s = s_x$, $e_s = e_x - 1$, $f_s = e_y$, $s_e = s_y$, $f_y \le e_e \le e_x - (p + 1)$, and $f_e = f_y$.
    \item $s_s = s_x$, $e_s = e_x - 1$, $f_s = e_y + 1$, $s_e \ne s_y$, $f_y \le e_e \le e_x - (p + 1)$, and $f_e = f_y$.
\end{enumerate}

\noindent \textbf{Lemma SETZ-4A1:} If $s_x \ne s_y$, $e_x = f_y + (p + 1)$, $f_x = e_y + p$, and $e_x = f_x$, then one of the following statements holds:
\begin{enumerate}
    \item $s_s = s_x$, $e_s = e_x - 1$, $e_x - (p - 1) \le f_s \le e_y + 1$, $s_e \ne s_y$, $f_y \le e_e \le e_x - (p + 1)$, and $f_e = f_y$.
\end{enumerate}

\noindent \textbf{Lemma SETZ-4B:} If $s_x \ne s_y$, $e_x > f_y + (p + 1)$, $f_x = e_y + (p + 1)$, and $e_x = f_x$, then one of the following statements holds:
\begin{enumerate}
    \item $s_s = s_x$, $e_s = e_x - 1$, $e_x - p \le f_s \le e_y + 1$, $s_e \ne s_y$, $f_y \le e_e \le e_x - (p + 2)$, and $f_e = f_y$.
\end{enumerate}

\end{document}